\documentclass[a4paper, 12pt]{amsart}

\usepackage{amssymb}
\usepackage{amsthm}
\usepackage{amsmath}

\usepackage{comment}

\usepackage{amscd}

\usepackage[mathscr]{eucal}

\usepackage{enumitem}

\theoremstyle{plain}
\newtheorem{thm}{Theorem}[section]
\newtheorem{prop}[thm]{Proposition}
\newtheorem{cor}[thm]{Corollary}
\newtheorem{lem}[thm]{Lemma}

\theoremstyle{definition}
\newtheorem{df}{Definition}[section]

\newtheorem{ques}[thm]{Question}
\newtheorem{conj}[thm]{Conjecture}

\theoremstyle{remark}
\newtheorem{rmk}{Remark}[section]

\newtheorem*{ac}{Acknowledgements}

\newenvironment{useofai}{%
  \medskip\noindent\textbf{Use of AI.}\ }{\par}

\newcommand{\zz}{\mathbb{Z}}
\newcommand{\qq}{\mathbb{Q}}
\newcommand{\rr}{\mathbb{R}}

\DeclareMathOperator{\card}{Card}

\newcommand{\supp}{\text{{\rm supp}}}

\newcommand{\yosub}{\subseteq}

\newcommand{\yogf}[1]{\mathbb{F}_{#1}}

\newcommand{\yocoef}[2]{\boldsymbol{C}(#1, #2)}

\newcommand{\yocase}[2]{Case #1.~[#2]:}

\newcommand{\yorcf}[2]{\mathfrak{K}(#1, #2)}
\newcommand{\yorcff}{\zeta}

\newcommand{\yovring}[2]{\mathfrak{A}(#1, #2)}
\newcommand{\yovideal}[2]{\mathfrak{o}(#1, #2)}

\newcommand{\youhaq}[4]{\mathbb{A}_{#3, #4}(#1, #2)}
\newcommand{\youhavq}[4]{U_{#1, #2,  #3, #4}}

\newcommand{\youit}{\tau}

\newcommand{\yoha}[2]{\mathbb{H}(#1, #2)}
\newcommand{\yohav}[2]{v_{#1, #2}}

\newcommand{\yopideal}[2]{\mathbb{N}_{#1, #2}}
\newcommand{\yopf}[3]{\mathbb{P}_{#3}(#1, #2)}
\newcommand{\yopfv}[3]{V_{#1, #2, #3}}

\newcommand{\yostan}[3]{\mathrm{St}_{#1, #2, #3}}

\newcommand{\yoproj}{\mathrm{Pr}}

\newcommand{\yova}{\xi}

\newcommand{\yowitto}[1]{\mathbb{W}(#1)}
\newcommand{\yowitt}[1]{\mathrm{Fr}\mathbb{W}(#1)}
\newcommand{\yowiv}[1]{w_{#1}}

\newcommand{\yorep}{J}
\newcommand{\yorepc}{C}

\newcommand{\yonabs}[3]{\lVert #3\rVert_{#1, #2}}
\newcommand{\yonval}[2]{v_{#1, #2}}

\newcommand{\yoratio}{\eta}

\newcommand{\yosmf}[1]{\boldsymbol{#1}}

\newcommand{\yosva}{\gamma}

\newcommand{\yosvaq}{\delta}

\newcommand{\yoadhf}[1]{\mathbf{AH}(#1)}

\newcommand{\yobtm}{\mathord{\circledcirc}}

\newcommand{\yopnbf}[1]{\qq_{#1}}

\newcommand{\yopnbv}[1]{v_{#1}}

\newcommand{\yogset}{\mathcal{G}}
\newcommand{\yoqpset}{\mathcal{CH}}

\newcommand{\yomainmap}{I}
\newcommand{\yosubmap}{H}

\newcommand{\yoesp}{E}
\newcommand{\yoedis}{D}
\newcommand{\yoept}{\varpi}

\newcommand{\yolga}{\Psi}

\newcommand{\yoelm}{M}

\newcommand{\yotei}{\rho}
\newcommand{\yoteiq}{\omega}
\newcommand{\id}{\mathrm{id}}
\newcommand{\yoincc}{\Gamma}

\newcommand{\yorepcard}{\varkappa}
\newcommand{\yoespcard}{\theta}

\makeatletter
\@addtoreset{equation}{section}

\makeatother

\begin{document}

\title[Embedding theorem]
{
A non-Archimedean 
Arens--Eells isometric embedding theorem
on
valued fields
}

\author[Yoshito Ishiki]
{Yoshito Ishiki}

\address[Yoshito Ishiki]
{\endgraf
Department of Mathematical Sciences
\endgraf
Tokyo Metropolitan University
\endgraf
Minami-osawa Hachioji Tokyo 192-0397
\endgraf
Japan
}
\email{ishiki-yoshito@tmu.ac.jp}

\date{\today}

\begin{abstract}
In 1959, Arens and Eells  proved that every metric space can be isometrically embedded into a normed  linear space  as  a closed subset.  In later years, in the paper on a short proof of the Arens--Eells theorem, Michael implicitly pointed out that  the Arens--Eells theorem follows from the statement that every metric space can be isometrically embedded into a normed  linear space  as a linearly  independent  subset. In this paper, we prove  a non-Archimedean analogue of the Arens--Eells isometric embedding theorem, which states that for every non-Archimedean valued field $K$, every ultrametric space can be isometrically embedded into a non-Archimedean valued  field that is a valued field  extension of $K$ such that the image of the embedding is algebraically independent over $K$. 
\end{abstract}


\maketitle

\section{Introduction}\label{sec:intro}
In 1956, 
Arens and Eells
\cite{MR81458}
established 
that 
for 
every metric space 
$(X, d)$ 
there exist a real normed linear space 
$(V, \|*\|)$ 
and an 
isometric embedding 
$I\colon X\to V$
 such that 
$I(X)$
 is closed in 
 $V$. 
In later years, 
in the paper
\cite{MR162222}
 on 
a short proof of the
 Arens--Eells theorem, 
Michael
implicitly 
pointed out that
the Arens--Eells theorem follows from 
the statement  that  for every metric space 
$(X, d)$, 
we can take a real normed  linear 
 space 
 $(V, \lVert *\rVert)$  
 and an isometric embedding 
$I\colon X\to V$ 
such that 
$I(X)$ 
is linearly independent. 
Indeed, 
assume that the  statement
above is true, 
let 
$Y$ 
be the 
completion of 
a given metric space 
$(X, d)$, 
and take 
a normed linear space 
$W$
and 
an isometric embedding 
$J\colon Y\to W$
such that 
$J(Y)$
 is  linearly independent 
in 
$W$. 
Next 
we  let
 $V$
 be the 
linear subspace of 
$W$ 
generated by 
$J(X)$
and put 
$I=J|_{X}$. 
Since 
$J(Y)$ 
is linearly independent, 
we have 
$I(X)=V\cap J(Y)$. 
By the completeness of 
$Y$, 
we see that 
$J(Y)$
is closed in 
$W$, 
and hence 
$I(X)$ is 
closed in 
$V$, 
which is 
nothing but 
the Arens--Eells theorem. 
This argument can be found 
in 
 \cite[Footnote (${}^{1}$)]{MR321026}, 
 and
 a
 similar argument 
  will be used in the proof of 
 one of 
 our 
 main results
 (see Theorem
  \ref{thm:19main2}).

A metric 
$d$ 
on 
$X$ 
is said to be 
an 
\emph{ultrametric} 
or 
a
 \emph{non-Archimedean metric} if 
$d(x, y)\le d(x, z)\lor d(z, y)$ 
for all 
$x, y, z\in X$, 
where 
$\lor$ 
stands for the maximum operator on 
$\rr$. 
A set 
$R$ 
is 
a \emph{range set} 
if 
$0\in R$ 
and 
$R\yosub [0, \infty)$. 
An ultrametric 
$d$ 
on 
$X$ 
is 
said to be 
\emph{R-valued} 
if 
$d(x, y)\in R$
 for all 
 $x, y\in X$.

Some authors have tried   to investigate  non-Archimedean 
analogues of  theorems on metric spaces such as 
the Arens--Eells theorem. 
Megrelishvili and Shlossberg
\cite{MR3067710} proved 
a non-Archimedean 
Arens--Eells theorem, 
which embeds ultrametric spaces 
into 
linear spaces  over 
$\yogf{2}$. 
In  
\cite[Theorem 4.3]{MR3503300}, 
as an improvement of their theorem, 
they proved a 
non-Archimedean 
Arens--Eells theorem
on 
ultra-normed 
linear spaces over arbitrary 
non-Archimedean valued fields.
In 
\cite[Theorem 1.1]{Ishiki2021ultra}, 
as a non-Archimedean analogue of the Arens--Eells theorem, 
the author of the present paper showed that 
for every range set 
$R$, 
for every integral domain 
$A$ 
with the trivial absolute value 
$\lvert *\rvert$
(i.e., 
$\lvert x\rvert=1$ 
for all 
$x\neq 0$), 
and 
for every 
$R$-valued 
ultrametric space 
$(X, d)$, 
there exist  an 
$R$-valued  
ultra-normed module 
$(V, \lVert * \rVert)$ 
over 
$(A, \lvert * \rvert)$ 
and 
an isometric embedding 
$I\colon X\to V$ 
such that 
$I(X)$ 
is closed and linearly independent 
over 
$(A, \lvert *\rvert)$. 
Using this embedding theorem, 
the author proved 
a
($R$-valued)
non-Archimedean 
analogue 
of 
the Hausdorff 
extension theorem 
of metrics
(see also
\cite{MR4335845}).

There are  other attempts   to 
construct a geometric  embedding 
from an ultrametric space
into a space with an
algebraic structure. 
In \cite{MR1354831}, 
Vestfrid
constructed 
an 
 embedding of  a separable  ultrametric space into 
 a separable 
Hilbert space 
(see also \cite{MR1619924}). 
Schikhof
\cite{MR748978}
established 
that 
every ultrametric space can 
be isometrically  embedded into 
a non-Archimedean valued field
using   Hahn fields, 
which generalize
 fields of formal power series
(see also \cite{MR2584966}). 
In 
\cite[Conjecture 5.34]{MR0425916}, 
Broughan raised the 
following conjecture.

\begin{conj}\label{conj:bbb}
Let 
$p$ 
be an odd prime and 
 $(X, d)$ 
 be an 
 $H_{p}$-valued 
 metric space, 
 where 
 $H_{p}=\{0\}\cup \{\, p^{n}\mid n\in \zz\, \}$. 
Then 
there exist 
a non-Archimedean Banach algebra 
$(B, \lVert *\rVert)$ 
over 
$\qq_{p}$ 
and an isometry 
$i\colon X\to B$. 
\end{conj}

Note that every 
$H_{p}$-valued
 metric space is 
an ultrametric space.

In this paper, 
 as a 
generalization of 
the Schikhof theorem
(\cite{MR748978}) and 
known 
non-Archimedean analogues
(\cite[Theorem 1.1]{Ishiki2021ultra}
and \cite[Theorem 4.3]{MR3503300}) of 
the Arens--Eells theorem, 
we prove that 
for every non-Archimedean 
valued field 
$(K, \lVert *\rVert_{K} )$, 
and 
for every metric space 
$(X, d)$, 
there exist a valued field 
$(L, \lVert*\rVert_{L})$
 and an isometric embedding 
$I\colon X\to L$
such that 
$L$
 is a valued field extension of 
 $K$, 
the set 
$I(X)$ 
is closed in 
$L$, 
and 
$I(X)$
 is 
algebraically independent 
over 
$K$
(Theorem \ref{thm:19main2}). 
Our proof can be considered as 
a sophisticated version of 
Schikhof's embedding method
\cite{MR748978}.
A key point of the proof of 
Theorem 
\ref{thm:19main2} is 
to use the notion of
$p$-adic Mal'cev--Neumann fields
($p$-adic Hahn fields), 
which 
was
 first introduced by 
Poonen \cite{MR1225257} 
as
$p$-adic 
analogues of ordinary 
Hahn fields.

As an application, 
we also give an
 affirmative 
solution of Conjecture \ref{conj:bbb}
(see Theorem \ref{thm:frw}).

After proving our main results, 
 we will 
 briefly  remark that Theorem
  \ref{thm:19main2} is still valid  for generalized ultrametric spaces whose distances take values in a  general linearly ordered set
 (Theorem \ref{thm:genmain}).

The paper is organized as follows. 
Section
 \ref{sec:pre} 
presents 
notions and notations of 
metric spaces and 
valued fields. 
We introduce 
Hahn fields and 
$p$-adic Hahn fields, 
which play
an important role 
in
the proofs of
 our main results. 
We also prepare 
some basic statements on 
valued fields and Hahn fields.
In Section 
\ref{sec:propf}, 
we 
show
 statements on 
algebraic independence 
in 
($p$-adic)
Hahn fields. 
Section 
\ref{sec:isom}
is devoted 
to  
proving 
Theorem 
\ref{thm:19main2}. 
We also 
provide 
an affirmative 
solution of 
Conjecture 
\ref{conj:bbb}. 
In Section 
\ref{sec:gen}, 
we discuss an analogue of Theorem 
\ref{thm:19main2} 
for generalized ultrametric spaces whose
 distances take  values in a general linearly ordered set. 

\begin{useofai}
OpenAI Codex was used in the preparation of this manuscript for language
editing,
\LaTeX{}
typesetting assistance, and exploration of proof
constructions.
The author
 takes 
 full
responsibility for the mathematical content and the final version of the
manuscript.

\end{useofai}

\begin{ac}

The author would like to thank 
Tomoki Yuji for helpful 
advice on algebraic arguments.

The author wishes to express his 
deepest gratitude 
to 
all  members of 
Photonics Control Technology Team (PCTT) in
 RIKEN, 
where
the majority  of the paper was written, 
for their invaluable   support. 
Special thanks are extended to 
  the 
Principal Investigator of PCTT,
Satoshi Wada for 
the encouragement and support that 
transcended disciplinary boundaries.

This work was partially supported by JSPS 
KAKENHI Grant Number 
JP24KJ0182. 
\end{ac}


\section{Preliminaries}\label{sec:pre}

\subsection{Generalities}\label{subsec:gen}
In this paper, 
we 
use the 
set-theoretic 
notations of ordinals. 
For example, 
for ordinals 
$\alpha$,  $\beta$, 
we
 have 
 $\beta<\alpha$
  if 
 and only if 
 $\beta\in \alpha$.

\subsubsection{Metric spaces}
For a metric space 
$(X, d)$,
and 
for a subset of
 $A$ 
 of 
$X$, 
and 
for
 $x\in X$, 
we define 
$d(A, x)=\inf\{\, d(a, x)\mid a\in A\, \}$. 
For  
$a\in X$ 
and 
$r\in (0, \infty)$, 
we denote by 
$B(a, r; d)$
the closed ball 
centered at
 $a$ 
 with radius 
$r$. 
We often simply  represent it 
as
 $B(a, r)$ 
when 
no confusion can arise. 
Similarly, we define the 
open ball 
$U(a, r)$.

The proofs of 
the next two lemmas 
 are presented in 
 Propositions 
 18.2, 
 and 
 18.4, 
in 
\cite{MR2444734}, 
respectively. 
Note that 
in 
Lemma 
\ref{lem:isosceles}, 
we do not assume that 
$d(x, y)\ge 0$ 
for all 
$x, y\in X$. 

\begin{lem}\label{lem:isosceles}
Let 
$X$ 
be a set, 
and 
$d\colon X\times X\to \rr$ 
be a symmetric function, 
i.e., $d(x, y)=d(y, x)$ 
for all 
$x, y\in X$. 
Then 
$d$ 
satisfies the strong 
triangle inequality: 
$d(x, y)\le d(x, z)\lor d(z, y)$ 
if and only if 
for every triple
$x, y, z\in X$, 
the inequality 
$d(x, z)<d(z, y)$ 
implies 
$d(z, y)=d(x, y)$. 
\end{lem}

\begin{lem}\label{lem:ultraopcl}
Let 
$(X, d)$ 
be a pseudo-ultrametric space, 
$a\in X$ 
and 
$r\in [0, \infty)$. 
Then
for every  
$q\in B(p, r)$, 
we have 
$B(p, r)=B(q, r)$. 
\end{lem}

\subsubsection{Valued rings}

Let 
$A$ 
be a commutative ring. 
We say that a function 
$v\colon A\to \rr\sqcup \{\infty\}$
is 
a
\emph{(additive) valuation}
if the following conditions are satisfied:
\begin{enumerate}

\item 
for every 
$x\in A$, 
we have 
$v(x)=\infty$ 
if and only if 
$x=0$; 

\item 
for every pair 
$x, y\in A$, 
we have 
$v(xy)=v(x)+v(y)$; 

\item 
for every pair 
$x, y\in A$, 
we have 
$v(x+y)\ge v(x)\land v(y)$, 
where 
$\land$ 
stands for the 
minimum operator on 
$\rr$. 
\end{enumerate}
If
 $A$ 
 is a field, 
then it is called a 
\emph{valued field}. 
Note that for every valued ring 
$(A, v)$, 
we can extend the valuation 
$v$ 
to the fractional field 
$K$ 
of 
$A$. 
Namely, for 
$x=b/a\in K$, 
where 
$a, b\in A$ 
and 
$a\neq 0$, 
we define 
$v(x)=v(b)-v(a)$. 
In this case, 
$v$ 
is well-defined and 
the pair 
$(K, v)$ 
naturally becomes 
a valued field. 
For example, 
for a prime
 $p$, 
we define 
the 
$p$-adic 
valuation 
$\yopnbv{p}$
on 
$\zz$
 by 
declaring that 
$\yopnbv{p}$
is the 
number of 
the factor 
$p$
in the 
prime 
factorization of 
$x$. 
The completion 
$\zz_{p}$
of
 $\zz$ 
with respect 
to 
$\yopnbv{p}$
is called 
the 
\emph{ring of $p$-adic integers}. 
The fractional 
field
$\yopnbf{p}$
 of
  $\zz_{p}$
is called 
the 
\emph{field of 
$p$-adic numbers}. 
 For more discussion on 
$p$-adic
 numbers, 
we refer the readers to 
\cite{MR1703500} 
and 
\cite{MR2444734}.

We say that a function 
$\lVert *\rVert \colon A\to [0, \infty)$
is 
a
\emph{(non-Archimedean) absolute value}
or 
\emph{multiplicative   valuation}
if the following conditions are satisfied:
\begin{enumerate}

\item 
for every 
$x\in A$, 
we have 
$\lVert x \rVert=0$ 
if and only if 
$x=0$; 

\item 
for every pair 
$x, y\in A$, 
we have 
$\lVert xy\rVert =\lVert x\rVert \cdot \lVert y\rVert$; 

\item 
for every pair 
$x, y\in A$, 
we have 
$\lVert x+y\rVert \le  \lVert x\rVert \lor  \lVert y\rVert$, 
where 
$\lor$ 
stands for the 
maximum  operator on 
$\rr$. 
\end{enumerate}

In what follows, for every 
$\yoratio\in (1, \infty)$, 
we consider  that 
$\yoratio^{-\infty}=0$
and 
$-\log_{\yoratio}(0)=\infty$. 
For a valuation 
$v$ 
on a ring 
$A$, 
and for a real  number 
$\yoratio\in (1, \infty)$, 
we define 
$\yonabs{v}{\yoratio}{x}=\yoratio^{-v(x)}$
and 
$\yonval{\lVert *\rVert}{\yoratio}(x)
=-\log_{\yoratio}(x)$.

Fix 
$\yoratio\in (1, \infty)$, 
valuations and absolute values on a ring 
$A$ 
are essentially 
equivalent. 
Namely, 
they 
correspond  to each other 
as follows.
\begin{prop}\label{prop:rho}
Let 
$A$  
be a commutative ring. 
Then the following statements are true: 
\begin{enumerate}[label=\textup{(\arabic*)}]

\item 
For every 
$\yoratio\in (1, \infty)$, 
and 
for every valuation 
$v$ 
on 
$A$, 
the function 
$\yonabs{v}{\yoratio}{*}\colon A\to [0, \infty)$ is 
an absolute value on 
$A$;

\item 
For 
$\yoratio\in (1, \infty)$, 
and for every absolute value 
$\lVert *\rVert$ 
on 
$A$, 
the function 
$\yonval{\lVert *\rVert}{\yoratio}\colon A\to \rr$ 
is a valuation on 
$A$. 
\end{enumerate}
\end{prop}

In this paper, 
we use both  valuations and absolute values on a ring
due to 
Proposition 
\ref{prop:rho}. 
Since we represent a valuation  
(resp.~an absolute value) 
as
a symbol like 
$v$, 
or 
$w$ 
(resp.~$\lvert*\rvert$
or  
$\lVert*\rVert$), 
the readers will be able to  distinguish them.

For a valued 
ring 
$(A, v)$, 
the set 
$\yovring{A}{v}=\{\, x\in K\mid 0\le v(x)\, \}$ 
becomes 
a ring, 
 and 
$\yovideal{A}{v}=\{\, x\in K\mid 0< v(x)\, \}$  
is 
a maximal ideal of 
$\yovring{A}{v}$. 
We denote by 
$\yorcf{A}{v}$
the field 
$\yovring{A}{v}/\yovideal{A}{v}$, 
and we call it 
 the 
\emph{residue class field
 of 
$(A, v)$}. 
We also denote by  
$\yorcff_{A}\colon \yovring{A}{v}\to \yorcf{A}{v}$
the canonical projection. 
We simply represent it as 
$\yorcff$ 
when no confusion
can arise. 
We say that a subset 
$\yorep$ 
of 
$K$
is 
a 
\emph{complete system of representatives of the residue class field
$\yorcf{A}{v}$} 
if 
$\yorep\yosub \yovring{A}{v}$, 
$0\in \yorep$, 
and 
$\yorcff|_{\yorep}\colon \yorep\to  \yorcf{A}{v}$ 
is 
bijective. 
Note that our definition requires 
$0\in \yorep$.

\begin{rmk}
Let 
$(A, v)$ 
be a valued ring. 
The following statements are 
 well-known. 
\begin{enumerate}

\item 
If 
$x\in A$
 satisfies 
 $v(x)=0$, 
then 
$x$
 is invertible in 
 $A$. 

\item 
If 
$K$ 
is a fractional ring of 
$A$, 
then 
$\yovring{K}{v}=A$. 
\end{enumerate}
\end{rmk}

\subsection{Constructions of valued  fields}\label{subsec:const}

In this section, 
we review 
some constructions of valued fields such as 
Hahn fields. 
For more discussion, 
see 
\cite{MR3782290} 
and 
\cite{MR4367483}. 
For most of 
the proofs
in this section, 
we refer to 
\cite{MR1225257}
and 
\cite{MR0554237}
(see also 
\cite{MR0886475}).


\subsubsection{Hahn rings and fields}

A non-empty subset 
$S$ 
of 
$\rr$ 
is 
said to be 
\emph{well-ordered} 
if 
every non-empty subset of 
$S$ 
has a minimum.

We denote by 
$\yogset$ 
the 
set of all subgroups
$G$
 of 
$\rr$
with 
$1\in G$
(equivalently, 
$\zz\yosub G\yosub \rr$). 
For the sake of convenience, 
we only consider the setting
where
 $G\in \yogset$
in this paper.

Now we  review 
the construction 
of the Hahn fields  in 
\cite{MR1225257}. 
Let 
$G\in \yogset$
and 
$A$ 
be a commutative ring. 
For a map 
$a\colon G\to A$, 
we define the support  
$\supp(a)$  
of 
$a$
by 
the set 
$\{\, x\in G\mid a(x)\neq 0\, \}$. 
We denote by 
$\yoha{G}{A}$
the set of all 
$a\colon G\to A$ 
such that 
$\supp(a)$ 
is  well-ordered. 
We often symbolically  represent 
$a\in \yoha{G}{A}$
as 
\[
a=\sum_{g\in G}a(g)t^{g}, 
\]
where $t$ is an indeterminate. 
For every pair 
$a, b \in \yoha{G}{A}$, 
we define 
$a+b$ 
by 
\[
(a+b)(x)=a(x)+b(x). 
\]
We also define 
the multiplication 
$ab\colon G\to A$ 
by 
\[
ab=\sum_{g\in G}\left(\sum_{i, j\in G, i+j=g}a_{i}b_{j}\right)t^{g}. 
\]
Define a valuation 
$\yohav{G}{A}$
on 
$\yoha{G}{A}$ 
by 
$\yohav{G}{A}(a)=\min \supp(a)$.
Since 
$\supp(a)$  
is well-ordered, 
the minimum 
$\min\supp(a)$ 
actually exists. 
Note that 
$A$
becomes 
 a subring 
of 
$\yoha{G}{A}$.

\begin{prop}\label{prop:prophahn}
Let 
$G\in \yogset$, 
and 
$\yosmf{k}$ 
be a field. 
Then 
$(\yoha{G}{\yosmf{k}}, \yohav{G}{\yosmf{k}})$ 
becomes 
a valued field and  it 
satisfies 
$\yorcf{\yoha{G}{\yosmf{k}}}{\yohav{G}{\yosmf{k}}}=\yosmf{k}$. 
\end{prop}
\begin{proof}
See 
\cite[Corollary 1]{MR1225257}. 
\end{proof}

We call 
$(\yoha{G}{A}, \yohav{G}{A})$ 
the 
\emph{Hahn ring  associated with 
$G$ and $A$}
and call
it 
 the 
 \emph{Hahn field} 
 if 
$A$ 
is a field. 
Note that 
in general, 
we can define the Hahn fields even if 
$G$ 
is a linearly ordered Abelian group
(see \cite{MR1225257}).


\subsubsection{The $p$-adic Hahn fields}
A 
$p$-adic 
analogue of the Hahn fields was
 first 
introduced in
 \cite{MR1225257}. 
Let us review 
a
construction. 
A field 
$\yosmf{k}$ 
of 
 characteristic 
$p$ 
is said to be 
\emph{perfect} 
if 
$p=0$,  
or 
$p>0$ 
and 
every element
of $\yosmf{k}$ 
has  a 
$p$-th 
root
in 
$\yosmf{k}$. 
The following proposition states the 
existence of rings of Witt vectors.
The proof is presented in 
\cite{MR0554237}.

\begin{prop}\label{prop:witt}
Let 
$\yosmf{k}$ 
be a perfect field of characteristic 
$p>0$. 
Then there exists a unique 
valued 
ring 
$(A, v)$ 
of characteristic 
$0$ 
equipped 
with a
valuation 
 $v$ 
 such
 that 
 $v(A)=\zz_{\ge 0}$, 
 $v$ 
 is complete, 
  and 
$\yorcf{A}{v}=\yosmf{k}$. 
\end{prop}

For each perfect field 
$\yosmf{k}$,
we denote  by 
$(\yowitto{\yosmf{k}}, \yowiv{\yosmf{k}})$
the valuation 
ring stated in Proposition 
\ref{prop:witt}
and 
denote by 
$\yowitt{\yosmf{k}}$
the fractional field  of 
$\yowitto{\yosmf{k}}$. 
We use the same symbol 
$\yowiv{\yosmf{k}}$ 
as 
 the  valuation on 
 $\yowitto{\yosmf{k}}$
induced by 
$\yowiv{\yosmf{k}}$ 
in the canonical way. 
The ring 
$(\yowitto{\yosmf{k}}, \yowiv{\yosmf{k}})$
is called the 
\emph{ring of Witt vectors associated with 
$\yosmf{k}$}. 
Notice that 
for every prime 
$p$, 
we have 
$\yowitto{\yogf{p}}=\zz_{p}$
and 
$\yowitt{\yogf{p}}=\yopnbf{p}$, 
and the valuation 
$\yowiv{\yogf{p}}$
coincides with 
the
 $p$-adic 
valuation
$\yopnbv{p}$.

The next proposition explains the 
concrete representation of  elements of 
a ring of Witt vectors.

\begin{prop}\label{prop:multimulti}
Let 
$\yosmf{k}$ 
be a perfect field of characteristic 
$p>0$. 
Then there uniquely 
exists a 
map 
$\yotei_{\yosmf{k}}\colon \yosmf{k}\to \yowitto{\yosmf{k}}$
 such that 
 the set
 \[
 \{\, \yotei_{\yosmf{k}}(a)\mid a\in \yosmf{k}\, \}
 \]
 is a complete system of 
 representatives, and
$\yotei_{\yosmf{k}}(ab)=\yotei_{\yosmf{k}}(a)\yotei_{\yosmf{k}}(b)$
for all 
$a, b\in \yosmf{k}$. 
In this case, 
for every
 $x\in \yowitto{\yosmf{k}}$, 
there uniquely exists a 
sequence 
$\{a_{i}\}_{i\in \zz_{\ge 0}}$
in 
$\yosmf{k}$
such that 
$x=\sum_{n\in \zz_{\ge 0}}\yotei_{\yosmf{k}}(a_{n})p^{n}$. 
\end{prop}
\begin{proof}
See 
Proposition 8 on   page 35  
and the argument on  page 37 
in 
\cite{MR0554237}. 
\end{proof}

\begin{prop}\label{prop:functorwitt}
Let 
$\yosmf{k}$ 
and 
$\yosmf{l}$ 
be perfect fields of 
characteristic 
$p>0$. 
For every homomorphism 
$\phi\colon \yosmf{k}\to \yosmf{l}$, 
there uniquely exists  a homomorphism 
$\yowitto{\phi}\colon \yowitto{\yosmf{k}}\to \yowitto{\yosmf{l}}$ 
such that 
$\yorcff_{\yowitto{\yosmf{l}}}\circ \yowitto{\phi}=\phi\circ \yorcff_{\yowitto{\yosmf{k}}}$. 
\[
  \begin{CD}
     \yowitto{\yosmf{k}} @>{\yowitto{\phi}}>> \yowitto{\yosmf{l}} \\
  @V{\yorcff_{\yowitto{\yosmf{k}}}}VV    @V{\yorcff_{\yowitto{\yosmf{l}}}}VV \\
     \yosmf{k}   @>{\phi}>>  \yosmf{l}
  \end{CD}
\]
Moreover, 
if 
$x=\sum_{n\in \zz_{\ge 0}}\yotei_{\yosmf{k}}(a_{n})p^{n}\in \yowitto{\yosmf{k}}$, 
where 
$a_{n}\in \yosmf{k}$, 
then 
we have 
\[
\yowitto{\phi}(x)
=\sum_{n\in \zz_{\ge 0}}\yotei_{\yosmf{l}}(\phi(a_{n}))p^{n}. 
\]
In particular, 
we have 
$\yowiv{\yosmf{l}}(\yowitto{\phi}(x))=\yowiv{\yosmf{k}}(x)$
for all 
$x\in \yowitto{\yosmf{k}}$. 
\end{prop}
\begin{proof}
See
\cite[Proposition 10 on   page 39]{MR0554237}. 
\end{proof}

\begin{rmk}
It is known that 
the construction of 
rings of Witt vectors 
is 
a functor
(see \cite[Page 39]{MR0554237}). 
\end{rmk}

Now we discuss
a 
$p$-adic 
analogue of Hahn fields, 
which is 
defined as
a quotient field of 
a Hahn ring. 
For 
$G\in \yogset$, 
and 
for a 
perfect 
field 
$\yosmf{k}$
 of  characteristic 
 $p>0$, 
we define 
a subset 
$\yopideal{G}{\yosmf{k}}$ 
of 
$\yoha{G}{\yowitto{\yosmf{k}}}$
by 
the 
set of all 
$\alpha=\sum_{g\in G}\alpha_{g}t^{g}\in 
\yoha{G}{\yowitto{\yosmf{k}}}$ 
such that 
$\sum_{n\in \zz}\alpha_{g+n}p^{n}=0$ 
in 
$\yowitto{\yosmf{k}}$ 
for every 
$g\in G$.

\begin{prop}\label{prop:p-adic}
For every 
$G\in \yogset$, 
and 
for every  
perfect field 
$\yosmf{k}$ 
of characteristic 
$p>0$,
the set 
$\yopideal{G}{\yosmf{k}}$ 
is 
an ideal of the ring 
$\yoha{G}{\yowitto{\yosmf{k}}}$, 
and 
$\yoha{G}{\yowitto{\yosmf{k}}}/\yopideal{G}{\yosmf{k}}$ is a 
field
(i.e., $\yopideal{G}{\yosmf{k}}$  is maximal).
\end{prop}
\begin{proof}
See 
\cite[Proposition 3 ]{MR1225257}
and 
\cite[Corollary 3]{MR1225257}. 
\end{proof}

For subsets $A$, 
$B$ of 
$\rr$, 
we define 
$A+B=\{\, a+b\mid a\in A, b\in B\, \}$. 

\begin{lem}\label{lem:repbeta}
Let 
$G\in \yogset$, 
$p$ 
be a prime, 
and 
$\yosmf{k}$
be a field 
of characteristic 
$p$. 
Let 
$\yorep\yosub \yowitto{\yosmf{k}}$ 
be a 
complete system 
 of representatives of the residue class field of 
$\yosmf{k}$. 
Then
 every element 
$\alpha=\sum_{g\in G}\alpha_{g}t^{g}\in 
\yoha{G}{\yowitto{\yosmf{k}}}$ 
is equivalent to 
an element 
$\beta=\sum_{g\in G}\beta_{g}t^{g}$
modulo 
$\yopideal{G}{\yosmf{k}}$, 
where 
$\beta_{g}$
 is in 
$\yorep$. 
In addition,
for every
 $x\in \yoha{G}{\yowitto{\yosmf{k}}}$, 
 the family
$\{\beta_{g}\}_{g\in G}$ 
is unique and 
$\supp(\beta)\yosub \supp(\alpha)+\zz_{\ge 0}$. 
\end{lem}
\begin{proof}
See 
\cite[Proposition 4]{MR1225257}. 
\end{proof}

\begin{df}\label{df:bbbb}
Based on Lemma 
\ref{lem:repbeta}, 
for a fixed  complete system 
$\yorep$ 
of 
representatives, 
and 
for every 
$a\in \yoha{G}{\yowitto{\yosmf{k}}}$, 
we denote by 
$\yostan{G}{\yosmf{k}}{\yorep}(a)$
 the 
standard representation of 
$a$ 
with respect to 
$\yorep$
stated in
 Lemma
  \ref{lem:repbeta}. 
In this case, 
we have 
$\supp(\yostan{G}{\yosmf{k}}{\yorep}(a))\yosub \supp(a)+\zz_{\ge 0}$. 
\end{df}

Let 
$\yopf{G}{\yosmf{k}}{p}$
denote 
 the 
 quotient 
field 
$\yoha{G}{\yowitto{\yosmf{k}}}/\yopideal{G}{\yosmf{k}}$, 
and 
let 
$\yoproj\colon \yoha{G}{\yowitto{\yosmf{k}}}\to 
\yopf{G}{\yosmf{k}}{p}$
denote the 
canonical 
projection.

\begin{prop}\label{prop:indS}
Let 
$G\in \yogset$, 
and 
$\yosmf{k}$ 
be a perfect field
of characteristic 
$p>0$. 
Take a complete system 
$\yorep\yosub \yowitto{\yosmf{k}}$
of representatives of 
the residue class field
$\yosmf{k}$. 
Then the following 
statements are true:
\begin{enumerate}[label=\textup{(\arabic*)}]

\item 
For every 
$z\in \yoha{G}{\yowitto{\yosmf{k}}}$, 
 the value 
$\min \supp(\yostan{G}{\yosmf{k}}{\yorep}(z))$ 
is 
 independent of 
the choice of 
$\yorep$. 

\item 

We define 
the map 
$\yopfv{G}{\yosmf{k}}{p}\colon\yopf{G}{\yosmf{k}}{p}\to G$
by 
\[
\yopfv{G}{\yosmf{k}}{p}(x)=\min \supp(\yostan{G}{\yosmf{k}}{\yorep}(z)), 
\]
where 
the point 
$z$ 
belongs to 
$\yoha{G}{\yowitto{\yosmf{k}}}$
and satisfies 
$\yoproj(z)=x$. 
Then 
$\yopfv{G}{\yosmf{k}}{p}$
is 
a valuation on 
$\yopf{G}{\yosmf{k}}{p}$. 
\end{enumerate}
\end{prop}
\begin{proof}
See \cite[Proposition 5]{MR1225257}. 
\end{proof}

Based on 
Proposition 
\ref{prop:indS}, 
we call 
the valued field 
$(\yopf{G}{\yosmf{k}}{p}, \yopfv{G}{\yosmf{k}}{p})$ 
the  
\emph{$p$-adic 
Mal'cev--Neumann field}
or 
\emph{$p$-adic Hahn field}. 
Notice that 
$(\yopf{\zz}{\yogf{p}}{p}, \yopfv{\zz}{\yogf{p}}{p})$ 
is nothing but 
the 
field
$(\yopnbf{p}, \yopnbv{p})$
 of 
$p$-adic 
numbers. 

To consider 
characteristics 
of a valued field and 
its residue class field, 
we supplementally 
define 
$\yoqpset$
by the 
 set of all  pairs
$(q, p)$ 
such that
$q$ 
and 
$p$ 
are 
$0$ 
or a prime 
satisfying either of the following conditions: 
\begin{enumerate}[label=\textup{(Q\arabic*)}]

\item\label{item:condq1} 
$q=p$,

\item\label{item:condq2} 
$q=0$ 
and 
$0<p$. 
\end{enumerate}
Note that 
$(q, p)\in \yoqpset$
 satisfies 
\ref{item:condq2}  if and only if 
$q\neq p$. 

In order to discuss 
$p$-adic and 
ordinary 
Hahn fields in 
a
unified manner, 
we make a
notation 
 as 
follows.

\begin{df}\label{df:uk}
Let 
$G\in \yogset$, 
$(q, p)\in \yoqpset$, 
and 
let 
$\yosmf{k}$
 be  a 
 perfect field 
 of characteristic 
 $p$. 
We define a
field 
$\youhaq{G}{\yosmf{k}}{q}{p}$
 by
\[
\youhaq{G}{\yosmf{k}}{q}{p}
=\begin{cases}
\yoha{G}{\yosmf{k}} & \text{if $q=p$;}\\
\yopf{G}{\yosmf{k}}{p} & \text{if $q\neq p$.}
\end{cases}
\]
We also define a valuation
$\youhavq{G}{\yosmf{k}}{q}{p}$ 
on 
$\youhaq{G}{\yosmf{k}}{q}{p}$
 by 
\[
\youhavq{G}{\yosmf{k}}{q}{p}
=\begin{cases}
\yohav{G}{\yosmf{k}} & \text{if $q=p$;}\\
\yopfv{G}{\yosmf{k}}{p} & \text{if $q\neq p$.}
\end{cases}
\]
\end{df}

A metric space
$(X, d)$
 is 
said to be 
\emph{spherically complete}
if 
for every 
 sequence of 
(closed or open) balls 
$\{B_{i}\}_{i\in \zz_{\ge 0}}$
with 
$B_{i+1}\yosub B_{i}$
 for all 
 $i\in\zz_{\ge 0}$, 
we have 
$\bigcap_{i\in \zz_{\ge 0}}B_{i}\neq \emptyset$.

\begin{prop}\label{prop:propppp}
Let 
$G\in \yogset$, 
$(q, p)\in \yoqpset$, 
 and 
$\yosmf{k}$ 
be a perfect field of characteristic 
$p$. 
Then the following statements are true: 
\begin{enumerate}[label=\textup{(\arabic*)}]

\item\label{item:gg1}
$\youhavq{G}{\yosmf{k}}{q}{p}(\youhaq{G}{\yosmf{k}}{q}{p})=G$;

\item\label{item:kk1}
$\yorcf{\youhaq{G}{\yosmf{k}}{q}{p}}{\youhavq{G}{\yosmf{k}}{q}{p}}=\yosmf{k}$;

\item\label{item:o1}
$(\youhaq{G}{\yosmf{k}}{q}{p}, \youhavq{G}{\yosmf{k}}{q}{p})$ 
is 
spherically complete. 
In particular, it is complete
as a metric space.  

\end{enumerate}
\end{prop}
\begin{proof}
The statements 
\ref{item:gg1} and 
\ref{item:kk1}
follow
from the 
construction of 
$(\youhaq{G}{\yosmf{k}}{q}{p}, 
\youhavq{G}{\yosmf{k}}{q}{p})$. 
The statement 
\ref{item:o1} 
is proven by 
\cite[Theorem 1]{MR1225257}, 
and 
\cite[Theorem 4]{MR0006161}
(see also \cite[Theorem 6.11]{MR3782290}). 
\end{proof}

Next we consider 
homeomorphic embeddings 
between 
$p$-adic 
or ordinary 
Hahn fields. 
We begin with 
ordinary ones. 
Let 
$G, H\in \yogset$ 
 with 
$G\yosub H$, 
and 
$A, B$ 
be commutative rings. 
We represent 
$\iota$ 
as the inclusion map 
$G\to H$. 
Let 
$\phi\colon A\to B$ 
be a 
ring 
homomorphism. 
For 
$x=\sum_{g\in G}x_{g}t^{g}\in \yoha{G}{A}$, 
we define 
$\yoha{\iota}{\phi}(x)\in  \yoha{H}{B}$ 
by 
$\yoha{\iota}{\phi}(x)=\sum_{h\in H}y_{h}t^{h}$, 
where 
\[
y_{h}
=
\begin{cases}
\phi(x_{h})  & \text{if $h\in G$;}\\
0 & \text{if $h\not \in G$.}
\end{cases}
\]
Then 
$\yoha{\iota}{\phi}\colon 
\yoha{G}{A}
\to \yoha{H}{B}$ 
becomes a map. 
If 
$G=H$, 
we simply 
write it as 
$\yoha{G}{\phi}$. 
Let us 
observe properties of 
$\yoha{\iota}{\phi}$. 

\begin{prop}\label{prop:hahnfuntor}
Let 
$A, B$ 
be commutative rings, 
and 
$\phi\colon A\to B$ 
be a 
ring homomorphism. 
Let 
$G, H\in \yogset$ 
such that 
$G\yosub H$, 
and denote by 
$\iota$ 
the inclusion map 
$G\to H$. 
Then the map 
$\yoha{\iota}{\phi}\colon \yoha{G}{A}\to \yoha{H}{B}$
 is 
a ring homomorphism
and satisfies
\[
\yorcff_{\yoha{H}{B}}\circ \yoha{\iota}{\phi}=
 \phi\circ \yorcff_{\yoha{G}{A}}
 \]
 on 
 $\yovring{\yoha{G}{A}}{\yohav{G}{A}}$. 
   \[
  \begin{CD}
    \yovring{\yoha{G}{A}}{\yohav{G}{A}} @>{\yoha{\iota}{\phi}}>> \yovring{\yoha{H}{B}}{\yohav{H}{B}} \\
  @V{\yorcff_{\yoha{G}{A}}}VV    @V{\yorcff_{\yoha{H}{B}}}VV \\
     \yosmf{k}   @>{\phi}>>  \yosmf{l}
  \end{CD}
\] 
\end{prop}
\begin{proof}
The lemma follows from the 
definitions of 
$\yoha{\iota}{\phi}$ 
and 
Hahn fields. 
\end{proof}

Next we 
discuss 
$p$-adic
 Hahn fields, 
which are
defined by 
quotient 
fields of Hahn rings.

\begin{prop}\label{prop:pfunctor}
Let
$G, H\in \yogset$
with $G\yosub H$, 
and let 
$\iota\colon G\to H$ 
be the inclusion map, 
 $\yosmf{k}$ 
 and 
$\yosmf{l}$ 
be perfect field 
with 
characteristic
 $p>0$ 
 and 
$\phi\colon \yosmf{k}\to \yosmf{l}$
 is 
a
homomorphism.
Then 
the homomorphism 
$\yoha{\iota}{\yowitto{\phi}}\colon 
\yoha{G}{\yowitto{\yosmf{k}}}
\to \yoha{H}{\yowitto{\yosmf{l}}}$ 
satisfies 
\[
\yoha{\iota}{\yowitto{\phi}}(\yopideal{G}{\yosmf{k}})\yosub \yopideal{H}{\yosmf{l}}.
\] 
In particular, 
the map 
$\yoha{\iota}{\yowitto{\phi}}$ 
induces 
a homomorphism
\[
\yopf{\iota}{\phi}{p}\colon 
\yopf{G}{\yosmf{k}}{p}
\to \yopf{H}{\yosmf{l}}{p}
\]
such that 
 $\yorcff_{\yopf{H}{\yosmf{l}}{p}}\circ \yopf{\iota}{\phi}{p}=\phi\circ \yorcff_{\yopf{G}{\yosmf{k}}{p}}$
 on 
 $\yovring{\yopf{G}{\yosmf{k}}{p}}{\yopfv{G}{\yosmf{k}}{p}}$. 
  \[
  \begin{CD}
    \yovring{\yopf{G}{\yosmf{k}}{p}}{\yopfv{G}{\yosmf{k}}{p}} @>{\yopf{\iota}{\phi}{p}}>> \yovring{\yopf{H}{\yosmf{l}}{p}}{\yopfv{H}{\yosmf{l}}{p}} \\
  @V{\yorcff_{\yopf{G}{\yosmf{k}}{p}}}VV    @V{\yorcff_{\yopf{H}{\yosmf{l}}{p}}}VV \\
     \yosmf{k}   @>{\phi}>>  \yosmf{l}
  \end{CD}
\]

\end{prop}
\begin{proof}

Take 
$x\in \yopideal{G}{\yosmf{k}}$ 
and 
put 
$x=\sum_{g\in G}x(g)t^{g}$. 
Then for every 
$g\in G$
we have 
$\sum_{n\in \zz}x(g+n)p^{n}=0$ 
in 
$\yowitto{\yosmf{k}}$. 
Note that 
for a fixed
member
 $g\in G$
  and 
for a sufficiently 
 large number 
  $m\in \zz_{\ge 0}$,
we have 
$x(g+n)=0$ for all 
$n<-m$
since 
$\{\, g\in G\mid x(g)\neq 0\, \}$
is well-ordered. 
By the strong triangle inequality, 
the limit 
$\sum_{n\in \zz}x(g+n)p^{n}=0$ 
implies that  
 $x(g+n)p^{n}\to 0$
 in 
$\yowitto{\yosmf{k}}$ as 
$n\to \infty$
(see \cite[Theorem 2.24]{MR3782290}). 
Since 
$\yowiv{\yosmf{l}}(\yowitto{\phi}(x))=
\yowiv{\yosmf{k}}(x)$
for all 
$x\in \yowitto{\yosmf{k}}$
(see Proposition \ref{prop:functorwitt}), 
 we 
also have 
\[
\yowitto{\phi}(x(g+n))p^{n}\to 0
\] 
in 
$\yowitto{\yosmf{l}}$
as 
$n\to \infty$. 
Thus 
the infinite sum 
$\sum_{n\in \zz}\yowitto{\phi}(x(g+n))p^{n}$ 
is convergent and  we have 
\[
\sum_{n\in \zz}\yowitto{\phi}(x(g+n))p^{n}
=\yowitto{\phi}\left (\sum_{n\in \zz}x(g+n)p^{n}\right)
=\yowitto{\phi}(0)
=0.
\] 
This
shows  that 
$\yoha{\iota}{\yowitto{\phi}}(x) \in \yopideal{H}{\yosmf{l}}$, 
and hence 
\[
\yoha{\iota}{\yowitto{\phi}}(\yopideal{G}{\yosmf{k}})\yosub \yopideal{H}{\yosmf{l}}.
\]
In particular,
the map 
$\yoha{\iota}{\yowitto{\phi}}$
 induces a map 
$\yopf{\iota}{\phi}{p}\colon \yopf{G}{\yosmf{k}}{p}\to \yopf{H}{\yosmf{l}}{p}$. 

Proposition  \ref{prop:hahnfuntor} implies that 
a map 
$\yoha{\iota}{\yowitto{\phi}}\colon  
\yoha{G}{\yowitto{\yosmf{k}}}\to 
\yoha{H}{\yowitto{\yosmf{l}}}$
satisfies 
$\yorcff\circ \yoha{\iota}{\yowitto{\phi}}=\phi\circ\yorcff$. 
Then 
 $\yorcff\circ \yopf{\iota}{\phi}{p}=\phi\circ \yorcff$. 
\end{proof}

Let 
$G, H\in \yogset$
with 
$G\yosub H$, 
$(q, p)\in \yoqpset$, 
$\yosmf{k}$ 
and 
$\yosmf{l}$ 
be 
perfect
 fields 
of characteristic 
$p$, 
and 
$\phi\colon \yosmf{k}\to \yosmf{l}$
be 
a homomorphism. 
Denote by 
$\iota\colon G\to H$
the inclusion map. 
We define 
\[
\youhaq{\iota}{\phi}{q}{p}
=
\begin{cases}
\yoha{\iota}{\phi} & \text{if $q=p$;}\\
\yopf{\iota}{\phi}{p} & \text{if $q\neq p$.}
\end{cases}
\]

Let 
$(K, v)$ 
and 
$(L, w)$
 be valued fields. 
 We say that 
 $(L, w)$ 
 is a 
\emph{valued field 
 extension}
  of 
 $(K, v)$ 
 as a valued field
 if 
 $K\yosub L$ 
 and 
 $w|_{K}=v$.

\begin{prop}\label{prop:embedcoef}
Let 
$G, H\in \yogset$
with 
$G\yosub H$, 
$(q, p)\in \yoqpset$, 
 $\yosmf{k}$ 
 and 
$\yosmf{l}$ 
be perfect fields
of 
characteristic 
$p$ 
and 
$\phi\colon \yosmf{k}\to \yosmf{l}$ 
be 
a
homomorphism.
Then the  map 
$\youhaq{\iota}{\phi}{q}{p}\colon 
\youhaq{G}{\yosmf{k}}{q}{p}
\to \youhaq{H}{\yosmf{l}}{q}{p}$
is a homomorphism
such that 
\[
\yorcff_{\youhaq{H}{\yosmf{l}}{q}{p}}\circ 
\youhaq{\iota}{\phi}{q}{p}=\phi\circ \yorcff_{\youhaq{G}{\yosmf{k}}{q}{p}}
\]
on 
 $\yovring{\youhaq{G}{\yosmf{k}}{q}{p}} {\youhavq{G}{\yosmf{k}}{q}{p}}$. 
 \[
  \begin{CD}
     \yovring{\youhaq{G}{\yosmf{k}}{q}{p}} {\youhavq{G}{\yosmf{k}}{q}{p}} @>{\youhaq{\iota}{\phi}{q}{p}}>> \yovring{\youhaq{H}{\yosmf{l}}{q}{p}} {\youhavq{H}{\yosmf{l}}{q}{p}} \\
  @V{\yorcff}VV    @V{\yorcff}VV \\
     \yosmf{k}   @>{\phi}>>  \yosmf{l}
  \end{CD}
\]
In addition,
if 
$q\neq p$, 
then
the following
statements  are true:
\begin{enumerate}[label=\textup{(\arabic*)}]

\item\label{item:zp1}
$(\youhaq{G}{\yosmf{k}}{q}{p}, \youhavq{G}{\yosmf{k}}{q}{p})$
is a valued field extension
of 
$(\yowitt{\yosmf{k}}, \yowiv{\yosmf{k}})$;

\item\label{item:zp2}
In particular, 
$(\youhaq{G}{\yosmf{k}}{q}{p}, \youhavq{G}{\yosmf{k}}{q}{p})$
is a valued field extension of 
$(\yopnbf{p}, \yopnbv{p})$. 

\end{enumerate}
\end{prop}
\begin{proof}
If 
$q=p$, 
then the former part follows from Proposition 
\ref{prop:hahnfuntor}. 
If 
$q\neq p$, 
then, by the definition of 
$\yoqpset$, 
we have 
$q=0$ 
and 
$p>0$. 
In this case, the former part follows from Proposition 
\ref{prop:pfunctor}.

We next prove the latter part. 
Assume that 
$q\neq p$. 
By the former part
of the proposition, 
since 
$\yowitt{\yosmf{k}}=\youhaq{\zz}{\yosmf{k}}{q}{p}$
and 
$\zz\yosub G$ 
(see the definition of
 $\yogset$), 
we can regard 
$\yowitt{\yosmf{k}}$
as a
subfield of 
$\youhaq{G}{\yosmf{k}}{q}{p}$. 
Namely, 
\ref{item:zp1}
is true. 
Similarly, 
by 
$\zz\yosub G$, 
and 
$\yogf{p}\yosub \yosmf{k}$, 
since 
$(\yopnbf{p}, \yopnbv{p})$
is 
equal to 
$(\youhaq{\zz}{\yogf{p}}{q}{p}, 
\youhavq{\zz}{\yogf{p}}{q}{p})$, 
the field
 $\youhaq{G}{\yosmf{k}}{q}{p}$
is a valued field extension of 
$(\yopnbf{p}, \yopnbv{p})$. 
This implies 
\ref{item:zp2}. 
\end{proof}

\begin{rmk}
Related to 
Proposition
 \ref{prop:embedcoef}, 
we make the next remarks. 
\begin{enumerate}[label=\textup{(\arabic*)}]

\item 
The map 
$\youhaq{\iota}{\phi}{q}{p}$ 
is injective 
since 
it is a
field  homomorphism.

\item 
 The construction of 
$\youhaq{G}{\yosmf{k}}{q}{p}$ 
is
a functor.

\item 
In contrast to 
Proposition 
\ref{prop:functorwitt}, 
the author does not know whether 
$\youhaq{\iota}{\phi}{q}{p}$ 
is a  unique homeomorphism 
such that  
$\yorcff\circ \youhaq{\iota}{\phi}{q}{p}=\phi\circ \yorcff$
 or not. 
\[
  \begin{CD}
     \yovring{\youhaq{G}{\yosmf{k}}{q}{p}} {\youhavq{G}{\yosmf{k}}{q}{p}} @>{\youhaq{\iota}{\phi}{q}{p}}>> \yovring{\youhaq{H}{\yosmf{l}}{q}{p}} {\youhavq{H}{\yosmf{l}}{q}{p}} \\
  @V{\yorcff}VV    @V{\yorcff}VV \\
     \yosmf{k}   @>{\phi}>>  \yosmf{l}
  \end{CD}
\]

\end{enumerate}
\end{rmk}

\begin{prop}\label{prop:teichpoonen}
Let 
$G\in \yogset$, 
$(q, p)\in \yoqpset$, 
and let 
$\yosmf{k}$ 
be a perfect field of characteristic 
$p$. 
Assume that 
$q\neq p$. 
By Proposition
 \ref{prop:embedcoef}, 
we can regard 
$\yowitto{\yosmf{k}}$ 
as a subring of 
$\yovring{\youhaq{G}{\yosmf{k}}{q}{p}}
{\youhavq{G}{\yosmf{k}}{q}{p}}$
in a canonical way. 
Then there uniquely exists 
a multiplicative map
\[
\yoteiq_{G,\yosmf{k}}\colon \yosmf{k}\to
\yovring{\youhaq{G}{\yosmf{k}}{q}{p}}
{\youhavq{G}{\yosmf{k}}{q}{p}}
\]
such that
$\yorcff\circ \yoteiq_{G,\yosmf{k}}=\id_{\yosmf{k}}$
and 
$\yoteiq_{G,\yosmf{k}}(\yosmf{k})\subset  \yovring{\yowitto{\yosmf{k}}}{\yowiv{\yosmf{k}}}$.
\end{prop}
\begin{proof}
We first construct 
$\yoteiq_{G,\yosmf{k}}$. 
By Proposition \ref{prop:embedcoef}, 
the inclusion
$\yowitto{\yosmf{k}}\subset \youhaq{G}{\yosmf{k}}{q}{p}$ 
induces an embedding 
\[
\yoincc\colon \yovring{\yowitto{\yosmf{k}}}{\yowiv{\yosmf{k}}}
\hookrightarrow
\yovring{\youhaq{G}{\yosmf{k}}{q}{p}}{\youhavq{G}{\yosmf{k}}{q}{p}}.
\] 
By Proposition
\ref{prop:multimulti}, 
there uniquely exists a multiplicative map
$\yotei_{\yosmf{k}}\colon \yosmf{k}\to \yowitto{\yosmf{k}}$
whose image 
$\yorep=\{\,\yotei_{\yosmf{k}}(a)\mid a\in \yosmf{k}\,\}$
is a complete system of representatives of the residue class field 
$\yosmf{k}$ 
in 
$\yowitto{\yosmf{k}}$. 
Hence 
$\yotei_{\yosmf{k}}$ 
induces a multiplicative map
\[
\yoteiq_{G,\yosmf{k}}=\yoincc\circ\yotei_{\yosmf{k}} \colon \yosmf{k}\to
\yovring{\youhaq{G}{\yosmf{k}}{q}{p}}
{\youhavq{G}{\yosmf{k}}{q}{p}}.
\]
By the construction, 
we see that 
$\yoteiq_{G,\yosmf{k}}(\yosmf{k})\subset  \yovring{\yowitto{\yosmf{k}}}{\yowiv{\yosmf{k}}}$.
The equality 
$\yorcff\circ \yoteiq_{G,\yosmf{k}}=\id_{\yosmf{k}}$
follows from Proposition
 \ref{prop:multimulti}
and Proposition 
\ref{prop:embedcoef}.

We next show the uniqueness. 
Let 
$\psi\colon \yosmf{k}\to
\yovring{\youhaq{G}{\yosmf{k}}{q}{p}}
{\youhavq{G}{\yosmf{k}}{q}{p}}$ 
be a multiplicative map
such that 
$\yorcff\circ \psi=\id_{\yosmf{k}}$
and 
$\psi(\yosmf{k})\subset  \yovring{\yowitto{\yosmf{k}}}{\yowiv{\yosmf{k}}}$.
Since 
$\psi(\yosmf{k})\subset  \yovring{\yowitto{\yosmf{k}}}{\yowiv{\yosmf{k}}}$, 
we may regard 
$\psi$ 
as a multiplicative map from 
$\yosmf{k}$ 
to 
$\yowitto{\yosmf{k}}$. 
Since 
$\yorcff\circ \psi=\id_{\yosmf{k}}$,
the image 
$\psi(\yosmf{k})$ 
is a complete system of representatives
of the residue class field 
$\yosmf{k}$ 
in 
$\yowitto{\yosmf{k}}$. 
By the uniqueness of 
$\yotei_{\yosmf{k}}$ 
stated in Proposition \ref{prop:multimulti}, 
we see that 
$\psi(a)=\yotei_{\yosmf{k}}(a)$
for every 
$a\in \yosmf{k}$. 
Thus, as a map into 
$\youhaq{G}{\yosmf{k}}{q}{p}$, 
we conclude that 
$\psi=\yoincc\circ \yotei_{\yosmf{k}}
=\yoteiq_{G,\yosmf{k}}$.
This completes the proof. 
\end{proof}

Based on 
Proposition \ref{prop:teichpoonen}, 
we call 
$\yoteiq_{G,\yosmf{k}}$
the 
\emph{Teichm\"{u}ller representative map}. 

\begin{df}\label{df:tau}
Let 
$G\in \yogset$, 
$(q, p)\in \yoqpset$, 
and 
let 
$\yosmf{k}$
 be  a 
 perfect field 
 of characteristic 
 $p$. 
We make the following 
assumptions and  definitions. 
\begin{enumerate}[label=\textup{(\arabic*)}]

\item\label{item:df:111}
In the rest of the  paper,  
whenever
  we take a 
complete
system 
$\yorep\yosub \youhaq{G}{\yosmf{k}}{q}{p}$ 
of representatives 
of the residue class field 
$\yosmf{k}$, 
in the case of 
$q=p$, 
we define 
$\yorep=\yosmf{k}$
using the fact that
$\yosmf{k}\yosub \yoha{G}{\yosmf{k}}$. 
In the case of 
$q\neq p$, 
using 
the 
Teichm\"{u}ller representative map
$\yoteiq_{G, \yosmf{k}}\colon \yosmf{k}\to \yovring{\youhaq{G}{\yosmf{k}}{q}{p}}
{\youhavq{G}{\yosmf{k}}{q}{p}}$
stated in 
Proposition 
\ref{prop:teichpoonen}, 
we put 
$\yorep=\yoteiq_{G, \yosmf{k}}(\yosmf{k})$.

\item\label{item:df:222}
We define an 
element 
$\youit\in \youhaq{G}{\yosmf{k}}{q}{p}$
as follows. 
If 
$q=p$, 
we define 
 $\youit$ 
 by 
an indeterminate as in the definition of Hahn fields. 
If 
$q\neq p$, 
we define 
$\youit=p \in \yopf{G}{\yosmf{k}}{p}$. 
In this case, 
every element of 
$\youhaq{G}{\yosmf{k}}{q}{p}$ 
can be 
represented as 
a power series of 
$\youit$
with 
 powers 
 taking values 
  in 
$G$. 

\item 
For 
$y\in \youhaq{G}{\yosmf{k}}{q}{p}$, 
and 
$g\in G$, 
if
 $q=p$, 
then 
we define 
$\yocoef{y}{g}$
by 
the coefficient of 
$\youit^{g}$
in the power series representation of  
$y$.  
If 
$q\neq p$, 
then 
$\yorep=\yoteiq_{G, \yosmf{k}}(\yosmf{k})$, 
and we define 
$\yocoef{y}{g}\in \yorep$
by 
the coefficient of 
$y$ 
with respect to 
$\youit^{g}(=p^{g})$. 
Notice that
we can represent 
$y=\sum_{g\in G}\yocoef{y}{g}\youit^{g}$. 
In any case, 
coefficients are 
unique 
according to the definition of Hahn fields and Lemma \ref{lem:repbeta}. 
Thus they are well-defined. 
\end{enumerate}
\end{df}
\begin{lem}\label{lem:teichcoefsubfield}
Let 
$G\in \yogset$, 
$(q,p)\in \yoqpset$, 
and let 
$\yosmf{m}$ 
and 
$\yosmf{l}$ 
be perfect fields of characteristic 
$p$ with 
$\yosmf{m}\yosub \yosmf{l}$. 
Regard 
$\youhaq{G}{\yosmf{m}}{q}{p}$ 
as a subfield of 
$\youhaq{G}{\yosmf{l}}{q}{p}$ 
by Proposition \ref{prop:embedcoef}. 
Let 
$\yorep_{\yosmf{m}}$
and 
$\yorep_{\yosmf{l}}$
be complete systems of 
representatives 
defined in Definition 
\ref{df:tau}
of 
$\youhaq{G}{\yosmf{m}}{q}{p}$
and 
$\youhaq{G}{\yosmf{l}}{q}{p}$, 
respectively. 
If 
$a\in \yorep_{\yosmf{l}}$
satisfies 
$\yorcff(a)\in \yosmf{m}$, 
then 
$a\in \yorep_{\yosmf{m}}$
 and  
$a\in \youhaq{G}{\yosmf{m}}{q}{p}$.
\end{lem}
\begin{proof}
In the case of 
$q=p$, 
we have 
$\yorep_{\yosmf{m}}=\yosmf{m}$
and 
$\yorep_{\yosmf{l}}=\yosmf{l}$, 
and 
$\yorcff$ 
is the inclusion map. 
Thus, the lemma is true. 
In the case of 
$q\neq p$, 
$\yorep_{\yosmf{m}}=\yoteiq_{G, \yosmf{m}}(\yosmf{m})$
and 
$\yorep_{\yosmf{l}}=\yoteiq_{G, \yosmf{l}}(\yosmf{l})$. 
Put 
$b=\yorcff(a)$. 
Since 
$a\in \yoteiq_{G,\yosmf{l}}(\yosmf{l})$, 
there exists 
$c\in \yosmf{l}$ 
such that 
$a=\yoteiq_{G,\yosmf{l}}(c)$. 
By Proposition 
\ref{prop:teichpoonen}, 
we have 
$c=\yorcff(\yoteiq_{G,\yosmf{l}}(c))
=\yorcff(a)=b$.
Thus 
$a=\yoteiq_{G,\yosmf{l}}(b)$.
By the assumption,
we have  
$b\in \yosmf{m}$. 
By Proposition 
\ref{prop:functorwitt}, 
the Teichm\"{u}ller
 representative maps are compatible
  with the inclusion
$\yosmf{m}\yosub \yosmf{l}$. 
Hence, under the inclusion 
$\youhaq{G}{\yosmf{m}}{q}{p}
\yosub
\youhaq{G}{\yosmf{l}}{q}{p}$,
we have 
$\yoteiq_{G,\yosmf{m}}(b)=\yoteiq_{G,\yosmf{l}}(b)$.
Therefore 
$a=\yoteiq_{G,\yosmf{m}}(b)\in \youhaq{G}{\yosmf{m}}{q}{p}$.
This proves the lemma.
\end{proof}

A group 
$G$ 
is said to be 
\emph{divisible} 
if 
for every 
$g\in G$ 
and for every 
$n\in \zz_{\ge 1}$
there exists 
$h\in G$ 
such that 
$g=n\cdot h$. 
\begin{prop}\label{prop:pembed}
Let 
$G\in \yogset$ 
be divisible,  
$(q, p)\in \yoqpset$, 
$\yosmf{k}$ 
be an algebraically closed field
of  characteristic 
$p$, 
and  
$(K, v)$ 
be a valued field
of characteristic
$q$
 such that 
$v(K)\yosub G\sqcup\{\infty\}$ 
and 
$\yorcf{K}{v}\yosub \yosmf{k}$. 
Assume moreover that, if 
$q\neq p$, 
then 
$v(p)=1$.
Then there exists a homomorphic  embedding 
$\phi\colon K\to \youhaq{G}{\yosmf{k}}{q}{p}$
such that 
$v(x)=\youhavq{G}{\yosmf{k}}{q}{p}(\phi(x))$
for all 
$x\in K$. 
Namely, 
the field 
$(K, v)$ 
can be 
regarded as a valued subfield of 
$(\youhaq{G}{\yosmf{k}}{q}{p}, \youhavq{G}{\yosmf{k}}{q}{p})$. 
\end{prop}
\begin{proof}
See
 \cite[Corollary 5]{MR1225257}. 
\end{proof}



\section{Algebraic independence over valued fields}\label{sec:propf}

First we 
remark that, 
for 
valued fields
$(K, v)$ 
and 
$(L, w)$
such that 
 $(L, w)$ 
 is a
 valued field 
  extension of 
 $(K, v)$, 
 there exists a canonical injective embedding 
from 
$\yorcf{K}{v}$ 
into 
$\yorcf{L}{w}$. 
Namely, 
we can regard 
$\yorcf{K}{v}$ 
as a subset of 
$\yorcf{L}{w}$
since 
 the inclusion map 
$\iota\colon K\to L$
 satisfies 
$\iota(\yovring{K}{v})\yosub \yovring{L}{w}$
and 
$\iota(\yovideal{K}{v})\yosub 
\yovideal{L}{w}$. 
Thus it 
naturally induces an 
homomorphic embedding  
$\yorcf{K}{v}\to \yorcf{L}{w}$.

Let 
$K$ 
and 
$L$ 
be fields with 
$K\yosub L$. 
A member 
$x$ 
of 
$L$ 
is said to be 
\emph{transcendental over
 $K$} 
if 
$x$ 
is not a
root of any non-trivial polynomial 
with coefficients in 
$K$. 
A subset 
$S$ 
of 
$L$ 
is 
said to be 
\emph{algebraically independent} if 
any finite collection 
$x_{1}, \dots, x_{n}$ 
in 
$S$ 
does not 
satisfy any non-trivial polynomial equation with 
coefficients in 
$K$. 
Note that 
a singleton 
$\{x\}$ 
of 
$L$ 
is 
algebraically independent over 
$K$ 
if and only if 
$x$ 
is transcendental over 
$K$.

\begin{lem}\label{lem:onetr}
Let 
$(K_{0}, v_{0})$ 
and 
$(K_{1}, v_{1})$ 
be valued fields. 
Assume that 
$(K_{1}, v_{1})$ 
is a
valued field  extension of 
$(K_{0}, v_{0})$. 
If 
$x\in \yovring{K_{1}}{v_{1}}$ 
is such that 
$\yorcff(x)$ 
is transcendental over 
$\yorcf{K_{0}}{v_{0}}$, 
then 
$x$ 
is transcendental over 
$K_{0}$. 
\end{lem}
\begin{proof}
The lemma follows from 
\cite[Theorem 3.4.2]{MR2183496}. 
\end{proof}

\begin{lem}\label{lem:trtr}
Let 
$(K_{0}, v_{0})$ 
and 
$(K_{1}, v_{1})$ 
be valued fields such that 
$(K_{1}, v_{1})$ 
is a valued field  extension of 
$(K_{0}, v_{0})$. 
Assume that  
$x_{1}, \dots, x_{n},  y\in K_{1}$
satisfy the following conditions:
\begin{enumerate}[label=\textup{(\arabic*)}]

\item\label{item:algind}
the set 
$\{x_{1}, \dots, x_{n}\}$ 
is 
 algebraically independent  over 
 $K_{0}$;

\item\label{item:gocha}
there exists a field  
$L$ 
satisfying 
$\{x_{1},\dots, x_{n}\}\yosub L$, 
$K_{0}\yosub L\yosub K_{1}$, 
and satisfying that 
 there exist 
 $z, c\in L$ 
 such   that 
 $c(y-z)\in \yovring{K_{1}}{v_{1}}$
 and 
$\yorcff(c(y-z))\in \yorcf{K_{1}}{v_{1}}$  
is 
transcendental over 
$\yorcf{L}{v_{1}|_{L}}$. 
\end{enumerate}
Then 
the set 
$\{x_{1}, \dots, x_{n}, y\}$
is 
 algebraically independent over 
$K_{0}$. 
\end{lem}
\begin{proof}
For the sake of contradiction, 
suppose that  
the set 
 $\{x_{1}, \dots, x_{n}, y\}$
is  not algebraically independent over 
$K_{0}$. 
From 
\ref{item:algind} and 
the fact that 
$\{x_{1}, \dots, x_{n}\}\yosub L$, 
it follows that 
$y$ 
is algebraic over 
$L$, 
and hence 
so is 
$c(y-z)$. 
Using  
Lemma 
\ref{lem:onetr}
 together with 
 \ref{item:gocha}, 
we see that 
$c(y-z)$ 
is transcendental over 
$L$. 
This is a contradiction. 
Therefore, 
the set
$\{x_{1}, \dots, x_{n}, y\}$
is 
 algebraically 
 independent  over 
$K_{0}$.
\end{proof}

Now we focus on 
algebraically independent 
elements in 
($p$-adic)
Hahn fields. 

\begin{lem}\label{lem:333}
Let 
$\yoratio\in (1, \infty)$, 
$G\in \yogset$, 
$(q, p)\in \yoqpset$, 
and 
let 
$\yosmf{k}$ 
be a perfect field
of 
characteristic 
$p$. 
Fix  
a complete 
system 
$\yorep\yosub \youhaq{G}{\yosmf{k}}{q}{p}$
of
representatives
of the residue class field 
$\yosmf{k}$. 
Assume that  
a
set 
$S$
of  
non-zero elements 
of  
$\youhaq{G}{\yosmf{k}}{q}{p}$
satisfies 
that:
\begin{enumerate}[label=\textup{(N\arabic*)}]

\item\label{item:condp}
for 
every
distinct 
 pair 
$x, y\in S$, 
and 
for every 
$g\in G$ 
satisfying that 
$g\in [\youhavq{G}{\yosmf{k}}{q}{p}(x-y), \infty)$,
 if either of
$\yocoef{x}{g}$
 and 
 $\yocoef{y}{g}$ 
 is non-zero, 
then 
$\yocoef{x}{g}\neq 
\yocoef{y}{g}$. 
\end{enumerate}
Then for every finite subset 
$A=\{z_{1}, \dots, z_{n}\}$
of 
$S$, 
there exist
 $i_{0}\in \{1, \dots, n\}$ 
 and 
$g_{0}\in G$ 
such that 
$\yocoef{z_{i_{0}}}{g_{0}}\neq 0$ 
and 
$\yocoef{z_{i_{0}}}{g_{0}}\neq \yocoef{z_{j}}{g_{0}}$ 
for all 
$j\in \{1, \dots, n\}$ with 
$j\neq i_{0}$. 
\end{lem}
\begin{proof}
Put 
\[
u=\min\{\, \youhavq{G}{\yosmf{k}}{q}{p}(z_{i}- z_{j})\mid i\neq j\, \},
\] 
and take a pair 
$\{i_{0}, n_{0}\}$ 
such that 
\[
u=\youhavq{G}{\yosmf{k}}{q}{p}(z_{i_{0}}- z_{n_{0}}). 
\]
Then either 
$\yocoef{z_{i_{0}}}{u}$ 
or  
$\yocoef{z_{n_{0}}}{u}$ is 
non-zero. 
We may assume that 
$\yocoef{z_{i_{0}}}{u}\neq 0$. 
Using  the condition 
 \ref{item:condp} and the minimality of 
 $u$, 
we have 
$\yocoef{z_{i_{0}}}{u}\neq \yocoef{z_{j}}{u}$ 
for all 
$j\in \{1, \dots,  n\}$ with 
$j\neq i_{0}$. 
Putting 
$g_{0}=u$ completes the proof. 
\end{proof}

Let
$G\in \yogset$, 
$(q, p)\in \yoqpset$, 
$\yosmf{k}$ 
be  
a perfect field 
of characteristic 
$p$, 
and 
$K$
be 
  a subfield 
of 
$\youhaq{G}{\yosmf{k}}{q}{p}$. 
Let 
$\yorep
\yosub \youhaq{G}{\yosmf{k}}{q}{p}$
be the complete system of representatives defined in Definition \ref{df:tau}. 
Let 
$\yoadhf{K}$ 
 denote 
the 
subfield 
of 
$\yosmf{k}$
generated by 
\[
\{\, \yorcff(\yocoef{x}{g})\mid x\in K,  g\in G\, \}
\]
over 
$\yorcf{K}{v}$. 
The definition of 
$\yoadhf{K}$ is 
``ad-hoc'', 
which means that  
it depends  not only on  information about 
$K$,  
but also on 
the inclusion map 
$K\to \youhaq{G}{\yosmf{k}}{q}{p}$. 
Namely, 
even if 
$K, L\yosub \youhaq{G}{\yosmf{k}}{q}{p}$
are isomorphic  to each other as fields, 
it can happen that 
$\yoadhf{K}\neq \yoadhf{L}$.

\begin{prop}\label{prop:fintr}
Let 
$G\in \yogset$, 
$(q, p)\in \yoqpset$, 
and
$\yosmf{k}, \yosmf{l}$ 
be  perfect fields 
of  
characteristic 
$p$ 
such that 
$\yosmf{k}\yosub\yosmf{l}$. 
Take  a subfield
$K$ 
of 
$\youhaq{G}{\yosmf{l}}{q}{p}$
such that 
$\yorcf{K}{v}\yosub\yosmf{k}$. 
Let 
$\yorep\yosub\youhaq{G}{\yosmf{l}}{q}{p}$ 
be the complete system of representatives defined in Definition \ref{df:tau}
of 
$\yosmf{l}$. 
Assume that 
a set
$S$
 of non-zero elements 
of 
$\youhaq{G}{\yosmf{l}}{q}{p}$
satisfies 
 the following
conditions:
\begin{enumerate}[label=\textup{(T\arabic*)}]

\item\label{item:condpttt}
for every
$x\in S$, 
and 
for every 
distinct pair 
$g, h\in G$, 
if 
both
$\yocoef{x}{g}$
and 
$\yocoef{x}{h}$
are non-zero, 
then 
we have 
$\yocoef{x}{g}\neq \yocoef{x}{h}$;

\item\label{item:condtxxxxx}
for every 
distinct
pair 
$x, y\in S$, 
and 
for every 
pair 
$g, h\in G$
with 
$\youhavq{G}{\yosmf{l}}{q}{p}(x-y)\le \min\{g, h\}$, 
if 
both
$\yocoef{x}{g}$
and 
$\yocoef{y}{h}$
are non-zero, 
then 
we have 
$\yocoef{x}{g}\neq \yocoef{y}{h}$;

\item\label{item:condalg}
the set 
\[
\{\, \yorcff(\yocoef{x}{g})
\mid 
x\in S, g\in G, \yocoef{x}{g}\neq 0\, \}
\]
is 
algebraically independent  over 
$\yoadhf{K}$. 

\end{enumerate}
Then 
$S$ 
is 
algebraically independent 
over 
$K$. 
\end{prop}
\begin{proof}
First observe that 
the condition 
\ref{item:condtxxxxx}
implies that 
$S$
 satisfies 
\ref{item:condp}
 in 
Lemma \ref{lem:333}.

Let 
$\youit$
 be 
the same element of 
$\youhaq{G}{\yosmf{l}}{q}{p}$
as  in 
Definition 
\ref{df:uk}. 
Take 
$n$-many 
distinct 
members 
$z_{1}, \dots, z_{n}$ 
in 
$S$.
Now we prove  
that
$\{z_{1}, \dots, z_{n}\}$ 
is 
algebraically 
independent 
over 
$K$ 
 by  induction on 
$n$.

In the case of  
$n=1$, 
we 
put
$z_{1}=\sum_{g\in G}\yocoef{z_{1}}{g}\youit^{g}$, 
take 
$u\in G$ 
such that
$\yocoef{z_{1}}{u}\neq 0$, 
and  
put 
\[
A=\{\, \yocoef{z_{1}}{g}\mid g\in G, g\neq u\, \}. 
\]
Due to  
the condition 
\ref{item:condpttt}, 
we have 
$\yocoef{z_{1}}{u}\not \in A$. 
Since 
$\yorep$ 
is a system of representatives, 
we also have 
$\yorcff(\yocoef{z_{1}}{u})\notin \yorcff(A)$.
Let 
$\yosmf{m}$ 
be the 
perfect subfield of 
$\yosmf{l}$ 
generated by 
$\yoadhf{K}\cup \yorcff(A)$, 
and 
put 
$L=\youhaq{G}{\yosmf{m}}{q}{p}$. 
Notice that 
$\youhaq{G}{\yosmf{m}}{q}{p}$
is a subfield 
of 
$\youhaq{G}{\yosmf{l}}{q}{p}$
(see Proposition 
\ref{prop:embedcoef}). 
The fact
that 
$\yoadhf{K}\yosub \yosmf{m}$
implies that 
$K\yosub L$. 
Put 
\[
f=z_{1}-\yocoef{z_{1}}{u}\youit^{u}.
\]
By Definition \ref{df:tau}, 
we have 
\[
f=\sum_{g\in G,\ g\neq u}\yocoef{z_{1}}{g}\youit^{g}.
\]
Since 
$\yorcff(\yocoef{z_{1}}{g})\in \yosmf{m}$
for every 
$g\in G$ 
with 
$g\neq u$, 
Lemma 
\ref{lem:teichcoefsubfield}
shows that 
$\yocoef{z_{1}}{g}\in L=\youhaq{G}{\yosmf{m}}{q}{p}$, 
and hence 
$f\in L$. 
By the definition of 
$f$, 
we have 
\[
\youit^{-u}(z_{1}-f)=\yocoef{z_{1}}{u}.
\]
By
the assumption 
\ref{item:condalg}
and 
$\yorcff(\yocoef{z_{1}}{u})\notin \yorcff(A)$,
we see that 
$\yorcff(\youit^{-u}(z_{1}-f))$
is 
transcendental over 
$\yosmf{m}$. 
Thus, according to Lemma \ref{lem:onetr}, 
we conclude that 
$\youit^{-u}(z_{1}-f)$ 
is transcendental 
over 
$L$. 
Hence 
$z_{1}$ 
is transcendental 
over 
$L$. 
In particular, 
$z_{1}$ is 
transcendental over 
$K$.

Now, 
we 
fix 
$k\in \zz_{\ge 1}$ 
and 
assume that 
the case of 
$n=k$
 is true. 
 We next  consider 
the case of 
$n=k+1$. 
Since 
$S$
satisfies the 
condition \ref{item:condp}, 
we can take 
$i_{0}\in \{1, \dots, n\}$
and 
$g_{0}\in G$ 
stated in 
Lemma 
\ref{lem:333}. 
We may assume that 
$i_{0}=k+1$
by renumbering the indices
 if necessary. 
Put 
\[
A=\{\, \yocoef{z_{i}}{g}\mid g\in G,  i\in \{1, \dots, k\}\, \}
\cup\{\,\yocoef{z_{k+1}}{g}\mid g\in G, g\neq g_{0} \, \}. 
\] 
According to 
\ref{item:condpttt}, 
we have 
\[
\yocoef{z_{k+1}}{g_{0}}\not\in \{\,\yocoef{z_{k+1}}{g}\mid g\in G, g\neq g_{0} \, \}. 
\]
We next show that 
\[
\yocoef{z_{k+1}}{g_{0}}\not\in 
\{\, \yocoef{z_{i}}{g}\mid g\in G,  i\in \{1, \dots, k\}\, \}. 
\]
Suppose, to the contrary, that there exist 
$i\in \{1, \dots, k\}$ 
and 
$g\in G$ 
such that 
$\yocoef{z_{i}}{g}=\yocoef{z_{k+1}}{g_{0}}$.
By Lemma
 \ref{lem:333}, 
we have 
$\yocoef{z_{i}}{g_{0}}\neq \yocoef{z_{k+1}}{g_{0}}$.
In particular, 
we have 
$g\neq g_{0}$. 
Moreover, 
we see that 
\[
\youhavq{G}{\yosmf{l}}{q}{p}(z_{i}-z_{k+1})\le g_{0}.
\]
If 
$\youhavq{G}{\yosmf{l}}{q}{p}(z_{i}-z_{k+1})\le \min\{g,g_{0}\}$,
then the condition \ref{item:condtxxxxx} gives a contradiction to the 
current hypothesis 
$\yocoef{z_{i}}{g}=\yocoef{z_{k+1}}{g_{0}}$. 
Thus we have 
\[
g<\youhavq{G}{\yosmf{l}}{q}{p}(z_{i}-z_{k+1}).
\]
Hence 
$\yocoef{z_{i}}{g}=\yocoef{z_{k+1}}{g}$.
Therefore 
\[
\yocoef{z_{k+1}}{g}=\yocoef{z_{k+1}}{g_{0}}\neq 0.
\]
This contradicts the condition 
\ref{item:condpttt}. 
Thus 
$\yocoef{z_{k+1}}{g_{0}}\not\in A$. 
Let 
$\yosmf{m}$
 be a perfect subfield of
  $\yosmf{l}$ 
  generated by 
$\yoadhf{K}\cup \yorcff(A)$, 
and 
put 
$L=\youhaq{G}{\yosmf{m}}{q}{p}$. 
Similarly to 
the case of 
$n=1$, 
we observe that 
$K\yosub L$. 
Since 
$\yorcff(A)\yosub \yosmf{m}$, 
applying  
Lemma 
\ref{lem:teichcoefsubfield}
to 
$\yocoef{z_{i}}{g}$ for all 
$g\in G$ 
and  all   
$i\in \{1, \dots, k\}$, 
we have 
$z_{i}\in L$
for all
 $i\in \{1, \dots, k\}$. 
Define 
\[
f=z_{k+1}-\yocoef{z_{k+1}}{g_{0}}\youit^{g_{0}}. 
\]
Since the coefficients of 
$f$ 
belong to 
$A$ 
and 
$\yorcff(A)\yosub \yosmf{m}$, 
Lemma \ref{lem:teichcoefsubfield}
 shows that 
 $f\in L=\youhaq{G}{\yosmf{m}}{q}{p}$.
By the definition of 
$f$, 
we 
also have
\[
\youit^{-g_{0}}(z_{k+1}-f)=\yocoef{z_{k+1}}{g_{0}}. 
\]
Thus
the condition 
\ref{item:condalg} shows that 
$\yorcff(\youit^{-g_{0}}(z_{k+1}-f))$ 
is 
transcendental over 
$\yosmf{m}$. 
By the induction hypothesis, 
$\{z_{1},\dots, z_{k}\}$ 
is algebraically independent over 
$K$.
Hence
Lemma 
\ref{lem:trtr}
implies  that 
the 
set 
$\{z_{1}, \dots, z_{k}, z_{k+1}\}$
is 
algebraically independent over 
$K$. 
This finishes the proof. 
\end{proof}

\begin{rmk}
Put 
$w=\yocoef{x}{g}$. 
The condition 
\ref{item:condpttt}
means that 
$\yocoef{x}{g}$
is equal to  zero,  or 
there is no 
$h\in G$ 
with
 $h\neq g$
  such that
$\yocoef{x}{h}=w$. 
\end{rmk}




\section{Isometric embeddings of ultrametric spaces}\label{sec:isom}
In this section, 
we prove 
our non-Archimedean analogue of the  
Arens--Eells theorem. 
As a consequence, 
we  give  an affirmative solution of 
Conjecture 
\ref{conj:bbb}.

\subsection{A non-Archimedean Arens--Eells theorem}\label{subsec:naae}

\subsubsection{Preparations}
This subsection is devoted to
 proving the 
 following technical theorem, 
 which plays a central role of
 our first main theorem. 
Our proof of the next theorem  can be considered as 
a sophisticated version of 
 the proof of the main theorem of 
\cite{MR748978}.

\begin{thm}\label{thm:main0}
Let
$\yoratio\in (1, \infty)$, 
$(q, p)\in \yoqpset$, 
$G\in \yogset$,
$\yosmf{k}$ 
be a field, 
and 
$\yosmf{l}$ 
be a perfect field of 
characteristic 
$p$. 
Fix a 
cardinal 
$\yorepcard$
  and 
let 
 $\yorep\yosub \youhaq{G}{\yosmf{l}}{q}{p}$ 
 be 
 the complete system of representatives defined in Definition
  \ref{df:tau}
 of the 
residue class field 
$\yosmf{l}$. 
Let 
$\yorepc$ 
be a subset of 
$\yorep$. 
Put 
$R=\{0\}\sqcup \{\, \yoratio^{-g}\mid g\in G\, \}$. 
Assume that  the following conditions
are
satisfied:
\begin{enumerate}[label=\textup{(A\arabic*)}]

\item\label{item:h1}
$\yosmf{l}$ 
is a field  extension of 
$\yosmf{k}$;

\item\label{item:h3}
the subset 
$\yorcff(\yorepc)$  
of 
$\yosmf{l}$
is  algebraically  independent over 
$\yosmf{k}$;

\item\label{item:h4}
$\card(\yorepc)=\yorepcard$. 
\end{enumerate}
Then 
for every
$R$-valued 
 ultrametric space 
$(X, d)$ 
with either 
$X=\emptyset$ 
or 
$\card(X\sqcup\{\ast\})\le \yorepcard$ 
for some 
$\ast\notin X$,
there exists a map 
$\yomainmap \colon X\to \youhaq{G}{\yosmf{l}}{q}{p}$ 
such that: 
\begin{enumerate}[label=\textup{(B\arabic*)}]

\item\label{item:connzero}
each 
$\yomainmap(x)$
 is 
non-zero;

\item\label{item:conisom}
the map 
$\yomainmap$ 
is an isometric embedding from  
$(X, d)$  
into 
the ultrametric space
$(\youhaq{G}{\yosmf{l}}{q}{p}, \yonabs{\youhavq{G}{\yosmf{l}}{q}{p}}{\yoratio}{*})$;

\item\label{item:conconseq2}
for every
$x\in X$, 
and for every distinct pair 
$g, h\in G$, 
if 
both
$\yocoef{\yomainmap(x)}{g}$
and 
$\yocoef{\yomainmap(x)}{h}$
are non-zero, 
then 
$\yocoef{\yomainmap(x)}{g}
\neq 
\yocoef{\yomainmap(x)}{h}$;

\item\label{item:conconseq}
for every pair 
$x, y\in X$, 
and for every  pair 
$g, h\in G$
with 
\[
\youhavq{G}{\yosmf{l}}{q}{p}(\yomainmap(x)-\yomainmap(y))\le \min\{g, h\}, 
\]
if 
both 
$\yocoef{\yomainmap(x)}{g}$
and 
$\yocoef{\yomainmap(y)}{h}$
are non-zero, 
then 
we have 
$\yocoef{\yomainmap(x)}{g}
\neq 
\yocoef{\yomainmap(y)}{h}$;

\item\label{item:concoef}
for every 
$x\in X$, 
the set 
\[
\{\, \yocoef{\yomainmap(x)}{g}\mid g\in G, \yocoef{\yomainmap(x)}{g}\neq 0\, \}
\] 
is contained in  
 $\yorepc$.

\end{enumerate}
\end{thm}


In this subsection, 
in what follows, 
we fix 
objects in the 
assumption of Theorem 
\ref{thm:main0}. 
We divide 
the proof of 
Theorem 
\ref{thm:main0}
into several 
 lemmas. 

In the following preparation, 
we assume that 
$X\neq \emptyset$.
Take 
$\yoept\notin X$, 
and put 
$\yoesp=X\sqcup\{\yoept\}$.
Put 
$\yoespcard=\card(\yoesp)$.
By the assumption, 
we have 
$\yoespcard\le \yorepcard$.
Choose a representation
\[
\yoesp=\{\, \yova_{\alpha}\mid \alpha<\yoespcard\, \}
\]
with 
$\yova_{0}=\yoept$.
Fix 
$r_{0}\in R\setminus \{0\}$
and 
$x_{0}\in X$, 
and define 
an 
$R$-valued 
ultrametric 
$\yoedis$ 
on 
$\yoesp$
by 
$\yoedis|_{X\times X}=d$
and 
$\yoedis(x, \yoept)=d(x, x_{0})\lor r_{0}$. 
Then 
$\yoedis$ 
is actually 
an ultrametric 
(see, for example, 
\cite[Lemma 5.1]{Ishiki2021ultra}). 
A one-point extension of a metric space 
is a traditional method 
to prove analogues of 
the Arens--Eells theorem.

Put 
$\yorepc=\{\, \yoelm_{\alpha}\mid \alpha<\yorepcard\, \}$.
For every 
$\beta\le \yoespcard$, 
we also put 
$\yoesp_{\beta}=\{\, \yova_{\alpha}\mid \alpha<\beta\, \}$
and 
$\yorepc_{\beta}=\{\, \yoelm_{\alpha}\mid \alpha<\beta\, \}$.

Fix 
$\lambda\le \yoespcard$. 
We say that 
a map 
$f\colon \yoesp_{\lambda}\to \youhaq{G}{\yosmf{l}}{q}{p}$ 
is 
\emph{well-behaved}
if 
the following conditions are 
satisfied: 
\begin{enumerate}[label=\textup{(C\arabic*)}]

\item\label{item:azero}
if 
$\lambda=0$, 
then 
$f\colon \yoesp_{0}\to \youhaq{G}{\yosmf{l}}{q}{p}$
is the empty map
and 
if 
$\lambda>0$, 
then 
$f$ 
satisfies 
$f(\yoept)=0$, 
where 
$0$ 
is the zero element 
of 
$\youhaq{G}{\yosmf{l}}{q}{p}$;

\item\label{item:aisom}
the map 
$f$
 is an isometric embedding from 
$\yoesp_{\lambda}$ 
into 
$\youhaq{G}{\yosmf{l}}{q}{p}$;

\item\label{item:aseqseq}
for every 
$x\in \yoesp_{\lambda}$, 
and for every distinct pair 
$g, h\in G$, 
if 
both 
$\yocoef{f(x)}{g}$
and 
$\yocoef{f(x)}{h}$
are non-zero, 
then 
$\yocoef{f(x)}{g}\neq \yocoef{f(x)}{h}$;

\item\label{item:conconseq2c}
for every pair 
$x, y\in \yoesp_{\lambda}$, 
and for every  pair 
$g, h\in G$
with 
\[
\youhavq{G}{\yosmf{l}}{q}{p}(f(x)-f(y))\le \min\{g, h\}, 
\]
if both
$\yocoef{f(x)}{g}$
and 
$\yocoef{f(y)}{h}$
are non-zero, 
then 
we have 
$\yocoef{f(x)}{g}
\neq 
\yocoef{f(y)}{h}$;

\item\label{item:acoef}
for every 
$\alpha<\lambda$, 
the set 
\[
\{\, \yocoef{f(\yova_{\alpha})}{g}\mid g\in G, 
\yocoef{f(\yova_{\alpha})}{g}\neq 0\, \}
\] 
is contained in  
 $\yorepc_{\alpha+1}$. 

\end{enumerate}

For an ordinal 
$\lambda\le \yoespcard +1$, 
a family 
$\{\yosubmap_{\alpha}\colon \yoesp_{\alpha}\to \youhaq{G}{\yosmf{l}}{q}{p}\}_{\alpha<\lambda}$ 
is said to be 
\emph{coherent} 
if 
the following condition is true: 
\begin{enumerate}[label=\textup{(Coh)}]

\item 
for every 
$\beta<\lambda$ 
and for every 
$\alpha<\beta$, 
we have 
$\yosubmap_{\beta}|_{\yoesp_{\alpha}}=\yosubmap_{\alpha}$. 
\end{enumerate}

For 
simplicity, 
we define 
an ultrametric 
$e$ 
on 
$\youhaq{G}{\yosmf{l}}{q}{p}$
 by 
 \[
 e(x, y)= \yonabs{\youhavq{G}{\yosmf{l}}{q}{p}}{\yoratio}{x-y}. 
 \]
 
We shall 
construct a 
coherent family 
$\{\yosubmap_{\alpha}\}_{\alpha\le \yoespcard}$
of 
well-behaved 
maps using transfinite 
recursion. 
Recall that 
for a subset 
$S$
of 
$E$, 
and for a point 
$p\in E$, 
we denote 
$D(S, p)$ by 
the distance between
 $S$ 
 and 
 $p$, 
i.e., 
$D(S, p)=\inf_{x\in S}D(x, p)$. 

We begin with the 
following
convenient 
criterion. 

\begin{lem}\label{lem:19:propertyp}
Fix 
$\lambda<\yoespcard$ 
with 
$\lambda\neq 0$, 
where we use the convention 
$-\log_{\yoratio}(0)=\infty$.
Put
$u=\yoedis(\yoesp_{\lambda}, \yova_{\lambda})$
and 
$m=-\log_{\yoratio}(u)$. 
Let 
$\yosubmap_{\lambda}\colon \yoesp_{\lambda}\to \youhaq{G}{\yosmf{l}}{q}{p}$  
be a
well-behaved map. 
Assume that 
an isometric embedding 
$\yosubmap_{\lambda+1}\colon \yoesp_{\lambda+1}\to \youhaq{G}{\yosmf{l}}{q}{p}$ 
satisfies
$\yosubmap_{\lambda+1}|_{\yoesp_{\lambda}}=\yosubmap_{\lambda}$
and 
the 
following property:
\begin{enumerate}[label=\textup{(P\arabic*)}]

\item\label{item:p1p1}
for every 
$g\in G\cap (m, \infty)$, 
we have 
$\yocoef{\yosubmap_{\lambda+1}(\yova_{\lambda})}{g}
=0$;

\item\label{item:p2p2}
if 
$m<\infty$
and 
$m\in G$, 
then 
$\yocoef{\yosubmap_{\lambda+1}(\yova_{\lambda})}{m}
\in \{0, \yoelm_{\lambda}\}$.

\end{enumerate}
Then 
the map 
$\yosubmap_{\lambda+1}$
satisfies 
the conditions 
\ref{item:aseqseq}--\ref{item:acoef}. 
\end{lem}
\begin{proof}
First we note 
that the following 
claim is true:
\begin{enumerate}[label=\textup{(CL)}]
\item\label{item:propertyp}
For every 
$a\in (u, \infty)$,
there exists 
$z\in \yoesp_{\lambda}$
such that 
\[
e(\yosubmap_{\lambda+1}(z), \yosubmap_{\lambda +1}(\yova_{\lambda}))<a.
\]
In  other words, 
for every 
$s\in (-\infty, m)$, 
there exists 
$z\in \yoesp_{\lambda}$
such that for every 
$n\le  s$, 
we 
have 
\[
\yocoef{\yosubmap_{\lambda+1}(z)}{n}
=
\yocoef{ \yosubmap_{\lambda+1}(\yova_{\lambda})}{n}. 
\]

\end{enumerate}

\textbf{Proof of \ref{item:aseqseq}}:
Next we prove 
\ref{item:aseqseq}. 
Owing to 
\ref{item:aseqseq}
for 
$\yosubmap_{\lambda}$,  
we only need to  confirm the 
case of 
$x=\yova_{\lambda}$. 
Take
arbitrary distinct pair  
$g, h\in G$, 
and assume that 
$\yocoef{\yosubmap_{\lambda+1}(\yova_{\lambda})}{g}$ and 
$\yocoef{\yosubmap_{\lambda+1}(\yova_{\lambda})}{h}$
are non-zero. 
Due to the  condition  
\ref{item:p1p1}, 
we may assume that 
 $g\le m$
 and 
$h\le m$.

 \yocase{1}{$g=m$ or $h=m$}
 By the assumption that
  $g=m$
   or 
   $h=m$, 
 the condition  
 \ref{item:p2p2}
  implies that 
 either 
 $\yocoef{\yosubmap_{\lambda+1}(\yova_{\lambda})}{g}$
 or 
  $\yocoef{\yosubmap_{\lambda+1}(\yova_{\lambda})}{h}$
  is  
  equal
  to 
  $\yoelm_{\lambda}$. 
 Using 
 \ref{item:acoef}
  for 
 $\yosubmap_{\lambda}$, 
  the other  belongs to  
  $\yorepc_{\lambda}$. 
By 
  $\yoelm_{\lambda}\not \in  \yorepc_{\lambda}$, 
  we have 
 $\yocoef{\yosubmap_{\lambda+1}(\yova_{\lambda})}{g}\neq \yocoef{\yosubmap_{\lambda+1}(\yova_{\lambda})}{h}$.

\yocase{2}{$g<m$ and $h<m$}
Put 
$s=\max\{g, h\}$. 
Then 
$s<m$ 
and 
the claim 
 \ref{item:propertyp}
 enables us to take 
 $z\in \yoesp_{\lambda}$
  such that 
$\yocoef{\yosubmap_{\lambda+1}(z)}{n}=\yocoef{\yosubmap_{\lambda+1}(\yova_{\lambda})}{n}$
for all 
$n\le s$. 
In particular, 
we have 
$\yocoef{\yosubmap_{\lambda+1}(z)}{g}=\yocoef{\yosubmap_{\lambda+1}(\yova_{\lambda})}{g}$
and 
$\yocoef{\yosubmap_{\lambda+1}(z)}{h}=\yocoef{\yosubmap_{\lambda+1}(\yova_{\lambda})}{h}$. 
Applying 
\ref{item:aseqseq}
for $\yosubmap_{\lambda}$
 to 
 $z$, 
 $g$ 
 and 
 $h$, 
we observe that 
$\yocoef{\yosubmap_{\lambda+1}(z)}{g}\neq \yocoef{\yosubmap_{\lambda+1}(z)}{h}$. 
Hence, 
we have 
$\yocoef{\yosubmap_{\lambda+1}(\yova_{\lambda})}{g}\neq \yocoef{\yosubmap_{\lambda+1}(\yova_{\lambda})}{h}$.

Therefore \ref{item:aseqseq} for 
$\yosubmap_{\lambda+1}$ 
is true.

\textbf{Proof of \ref{item:conconseq2c}}:
Owing to 
\ref{item:conconseq2c}
for 
$\yosubmap_{\lambda}$,  
we only need to  confirm 
the case of 
$x=\yova_{\lambda}$ 
and 
$y\in \yoesp_{\lambda}$. 
Take
an
arbitrary  pair  
$g, h\in G$
with 
\[
\youhavq{G}{\yosmf{l}}{q}{p}(\yosubmap_{\lambda+1}(x)-\yosubmap_{\lambda+1}(y))\le \min\{g, h\}.
\] 
Assume that 
$\yocoef{\yosubmap_{\lambda+1}(\yova_{\lambda})}{g}$
and 
$\yocoef{\yosubmap_{\lambda+1}(y)}{h}$
are non-zero. 
Due to the condition  
\ref{item:p1p1}, 
we may assume that 
 $g\le m$.

 \yocase{1}{$g=m$}
The condition 
\ref{item:p2p2} shows that 
$\yocoef{\yosubmap_{\lambda+1}(\yova_{\lambda})}{g}=\yoelm_{\lambda}$. 
Using 
 \ref{item:acoef} 
 for 
 $\yosubmap_{\lambda}$, 
 we have 
 $\yocoef{\yosubmap_{\lambda+1}(y)}{h}\in \yorepc_{\lambda}$. 
By 
$\yoelm_{\lambda}\not \in \yorepc_{\lambda}$, 
we
obtain 
$\yocoef{\yosubmap_{\lambda+1}(\yova_{\lambda})}{g}\neq \yocoef{\yosubmap_{\lambda+1}(y)}{h}$.

\yocase{2}{$g<m$}
Define 
$s\in G$ 
by 
\[
s=
\begin{cases}
g & \text{if $h\ge m$};\\
\max\{g, h\} & \text{if $h<m$}. 
\end{cases}
\]
Then  
$s\in (-\infty, m)$. 
Thus the property 
\ref{item:propertyp}
enables us to take 
$z\in \yoesp_{\lambda}$
such that 
$\yocoef{\yosubmap_{\lambda+1}(z)}{n}=\yocoef{\yosubmap_{\lambda+1}(\yova_{\lambda})}{n}$
for all
$n\le s$. 
In this situation,
by the definition of 
$\youhavq{G}{\yosmf{l}}{q}{p}$, 
we have 
\[
s< \youhavq{G}{\yosmf{l}}{q}{p}(\yosubmap_{\lambda+1}(z)-\yosubmap_{\lambda+1}(\yova_{\lambda})).
\] 
Since 
$\min\{g, h\}\le s$
 in any case, 
we see that 
\[
-\youhavq{G}{\yosmf{l}}{q}{p}(\yosubmap_{\lambda+1}(z)-\yosubmap_{\lambda+1}(\yova_{\lambda}))
<-\youhavq{G}{\yosmf{l}}{q}{p}(\yosubmap_{\lambda+1}(\yova_{\lambda})-\yosubmap_{\lambda+1}(y))
\]
In this situation, we can apply 
Lemma 
\ref{lem:isosceles}
to   
$-\youhavq{G}{\yosmf{l}}{q}{p}$ 
and 
three points 
$\yosubmap_{\lambda+1}(\yova_{\lambda})$, 
$\yosubmap_{\lambda+1}(y)$, 
and 
$\yosubmap_{\lambda+1}(z)$. 
Then 
we obtain 
\[
-\youhavq{G}{\yosmf{l}}{q}{p}(\yosubmap_{\lambda+1}(z)-\yosubmap_{\lambda+1}(y))
=-\youhavq{G}{\yosmf{l}}{q}{p}(\yosubmap_{\lambda+1}(\yova_{\lambda})-\yosubmap_{\lambda+1}(y)). 
\]
Thus we also obtain 
$\youhavq{G}{\yosmf{l}}{q}{p}(\yosubmap_{\lambda+1}(z)-\yosubmap_{\lambda+1}(y))\le \min\{g, h\}$. 
Applying 
\ref{item:conconseq2c}
 for 
$\yosubmap_{\lambda}$
 to 
$y$ 
and 
$z$, 
we see that 
$\yocoef{\yosubmap_{\lambda+1}(z)}{g}\neq \yocoef{\yosubmap_{\lambda+1}(y)}{h}$. 
Hence, 
we have 
\[
\yocoef{\yosubmap_{\lambda+1}(\yova_{\lambda})}{g}\neq \yocoef{\yosubmap_{\lambda+1}(y)}{h}.
\] 

Therefore 
\ref{item:conconseq2c} is 
valid.

\textbf{Proof of \ref{item:acoef}}:
By \ref{item:propertyp}, 
if 
$g<m$
we have 
$\yocoef{\yosubmap_{\lambda+1}(\yova_{\lambda})}{g}
\in \yorepc_{\lambda}$. 
If 
$m<g$, 
then 
the condition 
\ref{item:p1p1}
implies that 
$\yocoef{\yosubmap_{\lambda+1}(\yova_{\lambda})}{g}=0$. 
Thus, 
by 
\ref{item:p2p2}, 
the set 
\[
\{\, \yocoef{\yosubmap_{\lambda+1}(\yova_{\lambda})}{g}\mid g\in G, \yocoef{\yosubmap_{\lambda+1}(\yova_{\lambda})}{g}\neq 0\, \}
\]
is contained in  
 $\yorepc_{\lambda}\sqcup\{\yoelm_{\lambda}\}=\yorepc_{\lambda+1}$. 
Namely, 
the condition 
\ref{item:acoef}
is true. 
\end{proof}

We next see the elementary lemma. 
\begin{lem}\label{lem:center}
Let 
$\youit$ 
be the same element 
of 
$\youhaq{G}{\yosmf{l}}{q}{p}$ 
as in 
Definition 
\ref{df:uk}. 
Take 
an 
arbitrary  sequence 
$\{a_{i}\}_{i\in \zz_{\ge 0}}$
in 
$\youhaq{G}{\yosmf{l}}{q}{p}$ 
and 
an arbitrary 
sequence 
$\{r_{i}\}_{i\in \zz_{\ge 0}}$
in 
$(0, \infty)$ 
such that
there exists 
$c>0$ 
with 
$c<r_{i}$
for all 
$i\in \zz_{\ge 0}$. 
Put 
\[
R=\bigcup_{i\in \zz_{\ge 0}}(-\infty, -\log_{\yoratio}(r_{i})).
\] 
If 
$\bigcap_{i\in \zz_{\ge 0}}B(a_{i}, r_{i}; e)\neq \emptyset$, 
then we can  take 
$\yosva\in \bigcap_{i\in \zz_{\ge 0}}B(a_{i}, r_{i}; e)$ 
such that 
\[
\yosva=\sum_{g\in G\cap R}s_{g}\youit^{g}, 
\]
where 
$s_{g}\in \yorep$. 
In this case, we have 
\[
\bigcap_{i\in \zz_{\ge 0}}B(a_{i}, r_{i}; e)=
\bigcap_{i\in \zz_{\ge 0}}B(\yosva, r_{i}; e).
\] 
\end{lem}
\begin{proof}
By the assumption, take a point
$a\in \youhaq{G}{\yosmf{l}}{q}{p}$
such that 
$a\in \bigcap_{i\in \zz_{\ge 0}}B(a_{i}, r_{i}; e)$. 
We put 
$a=\sum_{g\in G}s_{g}\youit^{g}$, 
where 
$s_{g}\in \yorep$ 
and 
define 
$\yosva$
as a slice of 
$a$ 
by 
$R$, 
i.e., 
$\yosva=\sum_{g\in G\cap R}s_{g}\youit^{g}$. 
By the definition of 
$e$, 
(or 
$\youhavq{G}{\yosmf{l}}{q}{p}$), 
we have 
$\yosva\in B(a, r_{i}; e)$
for all 
$i\in \zz_{\ge 0}$. 
Thus 
Lemma 
\ref{lem:ultraopcl} 
proves the lemma. 
\end{proof}

\begin{rmk}
Since now we only consider metrics valued in 
$\rr$, 
the set 
$R$
 is equal to 
$(-\infty, -\log_{\yoratio}(r))$, 
where 
$r=\inf_{i\in\zz_{\ge 0}}r_{i}$. 
To extend our theory to generalized metrics, 
we claim the statement in a slightly general form. 
\end{rmk}

Next we show  lemmas 
corresponding to 
steps of isolated ordinals 
in transfinite induction. 

\begin{lem}\label{lem:0dis}
Fix 
$\lambda<\yoespcard$ 
with 
$\lambda\neq 0$, 
and
let 
$\yosubmap_{\lambda}\colon \yoesp_{\lambda}\to \youhaq{G}{\yosmf{l}}{q}{p}$  
be a
well-behaved map. 
If 
$\yoedis(\yoesp_{\lambda}, \yova_{\lambda})>0$, 
then we can obtain 
a well-behaved isometric embedding 
$\yosubmap_{\lambda+1}\colon \yoesp_{\lambda+1}\to \youhaq{G}{\yosmf{l}}{q}{p}$ 
such that 
$\yosubmap_{\lambda+1}|_{\yoesp_{\lambda}}=\yosubmap_{\lambda}$. 
\end{lem}
\begin{proof} 
Put 
$Y_{\lambda}=\yosubmap_{\lambda}(\yoesp_{\lambda})$. 
Put  
$u=\yoedis(\yoesp_{\lambda}, \yova_{\lambda})>0$
and
 $m=-\log_{\yoratio}(u)$. 
Let 
$\youit$ 
be the same element in 
$\youhaq{G}{\yosmf{l}}{q}{p}$
as in 
Definition 
\ref{df:uk}.

\yocase{1}{There is no $a\in \yoesp_{\lambda}$ 
such that 
$\yoedis(a, \yova_{\lambda})=u$}
Take a sequence 
$\{y_{i}\}_{i\in \zz_{\ge 0}}$ 
in 
$\yoesp_{\lambda}$ 
such that 
$\yoedis(y_{i+1}, \yova_{\lambda})<\yoedis(y_{i}, \yova_{\lambda})$ 
for all 
$i\in \zz_{\ge 0}$ 
and 
$\yoedis(y_{i}, \yova_{\lambda})\to u$ 
as 
$i\to \infty$. 
Put 
$r_{i}=\yoedis(y_{i}, \yova_{\lambda})$. 
In this situation, 
we see that 
$r_{i+1}<r_{i}$ 
and 
\[
B(\yosubmap_{\lambda}(y_{i+1}), r_{i+1}; e)\yosub 
B(\yosubmap_{\lambda}(y_{i}), r_{i}; e)
\]
for all 
$i\in \zz_{\ge 0}$
(see Lemma 
\ref{lem:ultraopcl}). 
Then, 
using the spherical completeness of 
$\youhaq{G}{\yosmf{l}}{q}{p}$
(\ref{item:o1} 
in 
Proposition 
\ref{prop:propppp}), 
we obtain 
\[
\bigcap_{i\in \zz_{\ge 0}}B(\yosubmap_{\lambda}(y_{i}), r_{i}; e)\neq \emptyset. 
\]
In this case, 
the set 
$\bigcap_{i\in \zz_{\ge 0}}B(\yosubmap_{\lambda}(y_{i}), r_{i}; e)$
is 
a closed ball of radius 
$r=\inf_{i\in \zz_{\ge 0}}r_{i}$ 
centered at some point in 
$\youhaq{G}{\yosmf{l}}{q}{p}$. 
Lemma 
\ref{lem:center} implies
that there exists
 $\yosva=
\sum_{g\in G\cap (-\infty, m)}s_{g}\youit^{g}$
in 
$\youhaq{G}{\yosmf{l}}{q}{p}$
such that 
\[
B(\yosva, r; e)=\bigcap_{i\in \zz_{\ge 0}}B(\yosubmap_{\lambda}(y_{i}), r_{i}; e).
\] 
We define 
a map 
$\yosubmap_{\lambda+1}\colon \yoesp_{\lambda+1}\to \youhaq{G}{\yosmf{l}}{q}{p}$ 
by 
$\yosubmap_{\lambda+1}|_{\yoesp_{\lambda}}=\yosubmap_{\lambda}$ 
and 
$\yosubmap_{\lambda+1}(\yova_{\lambda})=\yosva$. 
This definition
implies the condition 
\ref{item:azero}. 
Next we verify 
the condition 
\ref{item:aisom}. 
It suffices to show that 
\[
\yoedis(x, \yova_{\lambda})=
e(\yosubmap_{\lambda+1}(x), \yosubmap_{\lambda+1}(\yova_{\lambda})). 
\]
Take 
$n\in \zz_{\ge 0}$ 
such that 
$\yoedis(y_{n}, \yova_{\lambda})<\yoedis(x, \yova_{\lambda})$. 
Then
Lemma 
\ref{lem:isosceles}
implies 
$\yoedis(x, \yova_{\lambda})=\yoedis(y_{n}, x)$, 
and hence 
$\yoedis(y_{n}, x)=e(\yosubmap_{\lambda+1}(y_{n}), \yosubmap_{\lambda+1}(x))$. 
By 
$\yosva\in B(\yosubmap_{\lambda+1}(y_{n}), r_{n}; e)$, 
we have 
\begin{align*}
e(\yosubmap_{\lambda+1}(\yova_{\lambda}), \yosubmap_{\lambda+1}(y_{n}))&\le 
\yoedis(\yova_{\lambda}, y_{n})
<
\yoedis(\yova_{\lambda}, x)\\
&=e(\yosubmap_{\lambda+1}(y_{n}), \yosubmap_{\lambda+1}(x)). 
\end{align*}
Thus, using 
Lemma 
\ref{lem:isosceles}
 again, 
 we have 
\[
e(\yosubmap_{\lambda+1}(y_{n}), \yosubmap_{\lambda+1}(x))=e(\yosubmap_{\lambda+1}(\yova_{\lambda}), \yosubmap_{\lambda+1}(x)),
\]
and hence
$e(\yosubmap_{\lambda+1}(\yova_{\lambda}), \yosubmap_{\lambda+1}(x))=
\yoedis(\yova_{\lambda}, x)$. 
By the construction, 
the map
$\yosubmap_{\lambda+1}$
satisfies the 
properties 
\ref{item:p1p1}--\ref{item:p2p2}. 
Thus, we see that 
$\yosubmap_{\lambda+1}$
satisfies the conditions 
\ref{item:aseqseq}--\ref{item:acoef}.

\yocase{2}{There exists  
$a\in \yoesp_{\lambda}$ 
such that 
$\yoedis(a, \yova_{\lambda})=u$}
Applying  Lemma
 \ref{lem:center}
to a sequence 
$\{r_{i}\}_{i\in \zz_{\ge 0}}$ 
with  
$r_{i}=u$, 
there exists 
\[
\yosva=\sum_{g\in G\cap (-\infty, m)}s_{g}\youit^{g}\in \youhaq{G}{\yosmf{l}}{q}{p}
\]
such that 
$B(\yosva, u; e)=B(\yosubmap_{\lambda}(a), u; e)$.  
We put 
\[
\yosvaq=\yoelm_{\lambda}\cdot \youit^{m}+\yosva=\yoelm_{\lambda}\cdot \youit^{m}+
\sum_{g\in G\cap (-\infty, m)}s_{g}\youit^{g}. 
\]
Then 
Lemma 
\ref{lem:ultraopcl}
shows that
$B(\yosvaq, u; e)=B(\yosubmap_{\lambda}(a), u; e)$.
We define 
a map 
$\yosubmap_{\lambda+1}\colon \yoesp_{\lambda+1}\to \youhaq{G}{\yosmf{l}}{q}{p}$
by 
$\yosubmap_{\lambda+1}|_{\yoesp_{\lambda}}=\yosubmap_{\lambda}$ 
and 
$\yosubmap_{\lambda+1}(\yova_{\lambda})=\yosvaq$. 
To verify  
the condition 
\ref{item:aisom}, 
we  show that 
$e(\yosubmap_{\lambda+1}(\yova_{\lambda}), \yosubmap_{\lambda+1}(x))=\yoedis(\yova_{\lambda}, x)$
for all 
$x\in \yoesp_{\lambda}$. 
If 
$x\not \in B(a, u; \yoedis)$, 
then 
we have 
$\yoedis(x, a)=\yoedis(x, \yova_{\lambda})$. 
Similarly, 
we have 
\[
e(\yosubmap_{\lambda+1}(x), \yosubmap_{\lambda+1}(a))=
e(\yosubmap_{\lambda+1}(x), \yosubmap_{\lambda+1}(\yova_{\lambda})). 
\]
Since 
\[
e(\yosubmap_{\lambda}(x), \yosubmap_{\lambda}(a))=
e(\yosubmap_{\lambda}(x), \yosubmap_{\lambda}(a))
=\yoedis(x, a), 
\]
we have 
\[
e(\yosubmap_{\lambda+1}(x), \yosubmap_{\lambda+1}(\yova_{\lambda}))=\yoedis(x, \yova_{\lambda}).
\] 
If 
$x\in B(a, u; \yoedis)$, 
then 
$\yoedis(x, \yova_{\lambda})\le 
\yoedis(x, a)\lor \yoedis(a, \yova_{\lambda})\le u$. 
By the definition of 
$u$, 
we have 
$\yoedis(\yova_{\lambda}, x)=u$. 
We also see that 
$\yosubmap_{\lambda+1}(x)\in B(\yosubmap_{\lambda+1}(a), u; e)$. 
Then we can represent 
\[
\yosubmap_{\lambda+1}(x)=
\yosva+
\sum_{g\in G\cap 
[m, \infty)}c_{g}\youit^{g}, 
\]
where 
$c_{g}\in \yorepc_{\lambda}\cup \{0\}$. 
Take 
$\alpha<\lambda$ 
with 
$\yova_{\alpha}=x$. 
Then 
the condition 
\ref{item:acoef} 
for 
$\yosubmap_{\lambda}$ 
implies that, 
for every 
$g\in G\cap [m,\infty)$, 
if 
$c_{g}\neq 0$, 
then 
$c_{g}\in \yorepc_{\alpha+1}$. 
Since 
$\yoelm_{\lambda}\not\in \yorepc_{\alpha+1}$
and 
$\yoelm_{\lambda}\neq 0$, 
we have 
$c_{m}\neq \yoelm_{\lambda}$.
From the definition of 
$\youhavq{G}{\yosmf{l}}{q}{p}$
and 
$\yoelm_{\lambda}\not\in \yorepc_{\alpha+1}$, 
it follows that
$e(\yosubmap_{\lambda+1}(x), \yosubmap_{\lambda+1}(\yova_{\lambda}))
=u$. 
Hence 
\[
e(\yosubmap_{\lambda+1}(x), \yosubmap_{\lambda+1}(\yova_{\lambda}))=
\yoedis(x, \yova_{\lambda}).
\] 
Similarly to 
Case 1, 
by the construction, 
we see that 
$\yosubmap_{\lambda+1}$
satisfies 
the 
properties 
\ref{item:p1p1}--\ref{item:p2p2}. 
Thus, we see that 
$\yosubmap_{\lambda+1}$
satisfies the conditions 
\ref{item:aseqseq}--\ref{item:acoef}. 
\end{proof}

\begin{lem}\label{lem:podis}
Fix 
$\lambda<\yoespcard$
with 
$\lambda\neq 0$, 
and
let 
$\yosubmap_{\lambda}\colon 
\yoesp_{\lambda}\to \youhaq{G}{\yosmf{l}}{q}{p}$  
be a
well-behaved map. 
If 
$\yoedis(\yoesp_{\lambda}, \yova_{\lambda})=0$, 
then 
there exists a well-behaved map 
$\yosubmap_{\lambda+1}\colon 
\yoesp_{\lambda+1}\to \youhaq{G}{\yosmf{l}}{q}{p}$ 
such that 
$\yosubmap_{\lambda+1}|_{\yoesp_{\lambda}}=\yosubmap_{\lambda}$. 
\end{lem}
\begin{proof}
We now define 
$\yosubmap_{\lambda+1}(\yova_{\lambda})$ 
as follows. 
Take 
a sequence 
$\{x_{i}\}$ in  $\yoesp_{\lambda}$
 such that 
$x_{i}\to \yova_{\lambda}$
as 
$i\to \infty$, 
and define  
$\yosubmap_{\lambda+1}(\yova_{\lambda})=\lim_{i\to \infty}\yosubmap_{\lambda}(x_{i})$. 
The value 
$\yosubmap_{\lambda+1}(\yova_{\lambda})$ 
is independent of 
 the choice of 
a sequence
 $\{x_{i}\}_{i\in \zz_{\ge 0}}$
 (or see 
\cite[Theorem 2]{MR0390999} and \cite{MR0969516}).

First we confirm that 
$\yosubmap_{\lambda+1}$ 
is 
well-behaved. 
Since 
$\yosubmap_{\lambda}$
satisfies 
the condition
\ref{item:azero}, 
so does 
$\yosubmap_{\lambda+1}$.

To prove \ref{item:aisom}, 
it is enough to 
show 
$\yoedis(x, \yova_{\lambda})=e(\yosubmap_{\lambda+1}(x), \yosubmap_{\lambda+1}(\yova_{\lambda}))$
for all 
$x\in \yoesp_{\lambda}$. 
Take  a sufficiently large 
$n\in \zz_{\ge 0}$
so that 
\[
\yoedis(x_{n}, \yova_{\lambda})<\yoedis(x, \yova_{\lambda})
\]
and 
\[
e(\yosubmap_{\lambda+1}(x_{n}), \yosubmap_{\lambda+1}(\yova_{\lambda}))<
e(\yosubmap_{\lambda+1}(x), 
\yosubmap_{\lambda+1}(\yova_{\lambda})). 
\]
Lemma 
\ref{lem:isosceles}
implies 
that 
\[
\yoedis(x, \yova_{\lambda})=\yoedis(x_{n}, x)
\]
and 
\[
e(\yosubmap_{\lambda+1}(x), 
\yosubmap_{\lambda+1}(\yova_{\lambda}))
=
e(\yosubmap_{\lambda+1}(x_{n}), 
\yosubmap_{\lambda+1}(x)).
\] 
Since 
$\yosubmap_{\lambda}$
 is an isometry and 
$\yosubmap_{\lambda+1}|_{\yoesp_{\lambda}}=\yosubmap_{\lambda}$, 
we have 
\[
\yoedis(x_{n}, x)=e(\yosubmap_{\lambda+1}(x_{n}), 
\yosubmap_{\lambda+1}(x)).
\] 
Therefore we 
conclude that 
the condition 
\ref{item:aisom}
is true.

By the construction, 
for every finite subset 
$F\yosub G$, 
there exists 
$i\in \zz_{\ge 0}$ 
such that 
\[
\yocoef{\yosubmap_{\lambda+1}(\yova_{\lambda})}{g}
=
\yocoef{\yosubmap_{\lambda}(x_{i})}{g}
\]
for every 
$g\in F$. 
In the notation of Lemma 
\ref{lem:19:propertyp}, 
we have 
\[
u=\yoedis(\yoesp_{\lambda}, \yova_{\lambda})=0
\]
and hence 
$m=\infty$ 
by our convention. 
Hence 
$\yosubmap_{\lambda+1}$
satisfies the properties 
\ref{item:p1p1}--\ref{item:p2p2}. 
Therefore, applying Lemma 
\ref{lem:19:propertyp},  we see that 
$\yosubmap_{\lambda+1}$
satisfies the conditions 
\ref{item:aseqseq}--\ref{item:acoef}. 
\end{proof}

Let us prove Theorem \ref{thm:main0}. 
\begin{proof}[Proof of Theorem \ref{thm:main0}]
The case 
$X=\emptyset$ 
is immediate because the empty map satisfies all the required conditions. 
Thus we assume that 
$X\neq \emptyset$.
The proof is based on  
\cite{MR748978}.
Using transfinite recursion, 
we first construct 
a coherent family 
$\{\yosubmap_{\mu}\colon \yoesp_{\mu}\to \youhaq{G}{\yosmf{l}}{q}{p}\}_{\mu\le \yoespcard}$
of 
 well-behaved maps. 
Fix 
$\mu\le \yoespcard$
and assume that we have already obtained 
a coherent family 
$\{\yosubmap_{\alpha}\}_{\alpha<\mu}$
of 
well-behaved maps. 
Now we construct 
$\yosubmap_{\mu}$ 
as follows. 
If 
$\mu=0$, 
then 
we define 
$\yosubmap_{0}$
as
the empty map. 
If 
$\mu=1$, 
then we have 
$\yoesp_{1}=\{\yova_{0}\}$ 
and 
we define 
$\yosubmap_{1}\colon \yoesp_{1}\to \youhaq{G}{\yosmf{l}}{q}{p}$
by 
$\yosubmap_{1}(\yova_{0})=0$. 
If 
$\mu=\lambda+1$ 
for some 
$\lambda<\mu$ 
with 
$\lambda\neq 0$, 
then 
using Lemmas 
\ref{lem:0dis} 
and 
\ref{lem:podis}, 
we  obtain a well-behaved map
$\yosubmap_{\lambda+1}\colon \yoesp_{\lambda+1}\to \youhaq{G}{\yosmf{l}}{q}{p}$
such that 
$\yosubmap_{\lambda+1}|_{\yoesp_{\lambda}}=\yosubmap_{\lambda}$. 
If 
$\mu$ 
is a limit ordinal, 
then we define 
$\yosubmap_{\mu}\colon \yoesp_{\mu}\to \youhaq{G}{\yosmf{l}}{q}{p}$
by 
$\yosubmap_{\mu}(x)=\yosubmap_{\alpha}(x)$, 
where 
$x\in \yoesp_{\alpha}$. 
In this case,  
$\yosubmap_{\mu}$ 
is  well-defined 
since the family 
$\{\yosubmap_{\alpha}\}_{\alpha<\mu}$ 
is coherent. 
Therefore, 
according to transfinite recursion, 
we obtain a well-behaved isometric 
embedding 
$\yosubmap_{\yoespcard}\colon \yoesp_{\yoespcard}\to \youhaq{G}{\yosmf{l}}{q}{p}$. 

In this case, note that 
$\yoesp_{\yoespcard}=\yoesp$. 
Now we define 
$\yomainmap \colon  X\to \youhaq{G}{\yosmf{l}}{q}{p}$ 
by 
$\yomainmap=\yosubmap_{\yoespcard}|_{X}$
(recall that 
$\yoesp=X\sqcup \{\yoept\}$). 
Due to 
\ref{item:azero},
we have 
$\yosubmap_{\yoespcard}(\yoept)=0$
and 
$\yoept\not \in X$. 
Thus 
the condition 
\ref{item:connzero} 
is true. 
Since 
$\yosubmap_{\yoespcard}$
satisfies 
\ref{item:aisom}--\ref{item:acoef}, 
the map 
$\yomainmap$ 
satisfies the 
conditions 
\ref{item:conisom}--\ref{item:concoef}. 
This finishes the proof of 
Theorem 
\ref{thm:main0}. 
\end{proof}


\subsubsection{The proof of the first   main result}

The following is 
our first main result.

\begin{thm}\label{thm:19main2}
Let 
$\yoratio\in (1, \infty)$, 
$(q, p)\in \yoqpset$, 
$\yorepcard$
 be a   cardinal, 
and 
$G\in \yogset$ 
be 
divisible. 
Put 
$R=\{\, \yoratio^{-g}\mid g\in G\sqcup\{\infty\}\,\}$. 
Take
an  arbitrary 
  valued field
  $(K, v)$  
of characteristic 
$q$
  such that 
$v(K)\yosub G\sqcup\{\infty\}$
and 
$\yorcf{K}{v}$
has characteristic 
$p$. 
Assume moreover that, if 
$q\neq p$, 
then 
$v(p)=1$.
Then there exists a  valued field 
$(L, w)$ 
such that: 
\begin{enumerate}[label=\textup{(L\arabic*)}]

\item\label{item:l1}
the field 
$(L, w)$ 
is a valued field   extension of 
$(K, v)$;

\item\label{item:l2}
$w(L)= G\sqcup\{\infty\}$;

\item\label{item:l3}
for each 
$R$-valued
 ultrametric 
$(X, d)$ 
with either 
$X=\emptyset$ 
or 
$\card(X\sqcup\{\ast\})\le \yorepcard$ 
for some 
$\ast\notin X$,
there exists 
an isometric 
embedding  
$\yomainmap\colon (X, d)\to (L, \yonabs{w}{\yoratio}{*})$
such that 
the set 
$\yomainmap(X)$ 
is algebraically  independent over 
$K$. 
\end{enumerate}
Moreover, 
for every 
$R$-valued 
ultrametric space
$(X, d)$, 
there exist a valued field 
$(F, u)$ 
and a map 
$\yomainmap\colon X\to F$
such that: 
\begin{enumerate}[label=\textup{(F\arabic*)}]

\item\label{item:f1}
the map 
$\yomainmap\colon (X, d)\to (F, \yonabs{u}{\yoratio}{*})$ is an isometric embedding;

\item\label{item:f2}
the field 
$(F,  u)$ 
is a valued field   extension of 
$(K, v)$;

\item\label{item:f3}
$u(F)\yosub G\sqcup \{\infty\}$;

\item\label{item:f4}
the set
$\yomainmap(X)$ 
is closed in 
$F$;

\item\label{item:f5}
the set 
$\yomainmap(X)$
 is algebraically independent 
over 
$K$; 

\item\label{item:f6}
if 
$(X, d)$
is complete, 
then 
$(F, u)$
can be chosen 
 to be complete.

\end{enumerate}
\end{thm}
\begin{proof}
Let 
$\yosmf{k}$ 
be 
the  algebraic 
closure   of 
$\yorcf{K}{v}$.
Notice that 
$\yosmf{k}$
 is perfect. 
Take a perfect field 
$\yosmf{l}$ 
such that the transcendental degree of 
$\yosmf{l}$ 
over 
$\yosmf{k}$
is 
$\yorepcard$, 
and take a 
transcendental basis 
$B$ of
$\yosmf{l}$
over 
$\yosmf{k}$. 
In this case, 
we have 
$\card(B)=\yorepcard$. 
Let
$\yorep\yosub 
 \youhaq{G}{\yosmf{l}}{q}{p}$ 
be  the complete system of representatives defined in Definition 
\ref{df:tau}
of 
$\yosmf{l}$. 
Notice that 
$B\yosub \yorcff(\yorep)$. 
Put
$\yorepc=(\yorcff|_{\yorep})^{-1}(B)$. 
By Proposition \ref{prop:pembed}, 
we can take a homomorphic embedding 
\[
\psi\colon K\to \youhaq{G}{\yosmf{k}}{q}{p}
\]
that is  
a valued field
 embedding 
from 
$(K, v)$ 
into 
$(\youhaq{G}{\yosmf{k}}{q}{p}, \youhavq{G}{\yosmf{k}}{q}{p})$.
Let 
$\phi\colon \yosmf{k}\to \yosmf{l}$ 
be
the inclusion map. 
Let 
$\iota=\id_G\colon G\to G$.
Using 
a homomorphic embedding 
$\youhaq{\iota}{\phi}{q}{p}\circ \psi\colon K\to  \youhaq{G}{\yosmf{l}}{q}{p}$, 
in what follows, 
we consider that 
$K$ 
is a 
subfield 
 of 
$\youhaq{G}{\yosmf{l}}{q}{p}$. 
In this case, 
we see that 
$\yoadhf{K}\yosub \yosmf{k}$. 
Put 
$L=\youhaq{G}{\yosmf{l}}{q}{p}$
and 
$w=\youhavq{G}{\yosmf{l}}{q}{p}$. 
Then 
$(L, w)$
satisfies 
the conditions 
\ref{item:l1}
and 
\ref{item:l2}
(see Proposition \ref{prop:propppp}). 

Now we prove 
\ref{item:l3}. 
Take an 
$R$-valued 
ultrametric space
$(X, d)$
 satisfying the cardinality assumption in
  \ref{item:l3}. 
Applying 
Theorem 
\ref{thm:main0}
to 
$\yorepc$, 
$\yorep$, 
$\yosmf{l}$,
$\yosmf{k}$,
and 
$(X, d)$
with either 
$X=\emptyset$ 
or 
$\card(X\sqcup\{\ast\})\le \yorepcard$ 
for some 
$\ast\notin X$,
we can 
take 
$\yomainmap\colon X\to L$
satisfying 
the conditions 
\ref{item:connzero}--\ref{item:concoef}. 
The condition 
\ref{item:conisom}
shows 
that 
$\yomainmap$
is isometric. 
Due to 
the
conditions 
\ref{item:connzero}, 
\ref{item:conconseq2},
and 
\ref{item:conconseq},
and due to 
the fact that
$\yoadhf{K}\yosub \yosmf{k}$,   
the set 
\[
\yomainmap(X)=\{\, \yomainmap(x)\mid x\in X\, \}
\]
satisfies the 
assumptions in Proposition  
\ref{prop:fintr}.  
Then,
according to 
Proposition  
\ref{prop:fintr}, 
we see that 
$\yomainmap(X)$ 
is algebraically independent 
 over 
$K$. 
This means that 
the condition 
\ref{item:l3} is true.

We next prove the latter part. 
Let 
$(Y, e)$ 
be the completion of 
$(X, d)$. 
Then 
$(Y, e)$  is 
$R$-valued 
(see \cite[(12) in Theorem 1.6]{MR3782290}). 
Applying the former part of the theorem with 
$\yorepcard=\card(Y\sqcup\{\ast\})$, 
where
$\ast\notin Y$,
we obtain 
a valued field 
$(L, w)$
and 
 an isometric embedding 
$\yomainmap\colon Y\to L$
such that 
$\yomainmap(Y)$ 
is algebraically independent over 
$K$. 
Since 
$(Y, e)$ 
is complete
and 
$\yomainmap$ 
is isometric, 
the set $\yomainmap(Y)$
is closed in 
$L$. 
Let 
$F$ 
be the subfield of 
$L$
generated by 
$\yomainmap(X)=\{\,\yomainmap(x) \mid x\in X\, \}$
over 
$K$ 
and 
put 
$u=w|_{F}$.
Since 
$\yomainmap(Y)$
 is algebraically independent
over 
$K$, 
we have  
\[
\yomainmap(Y)\cap F=\yomainmap(X). 
\]
Thus 
$\yomainmap(X)$ is 
closed in 
$F$. 

To verify 
the condition 
\ref{item:f6}, 
we assume that 
$(X, d)$
is complete. 
The field 
$F$
constructed above is 
not necessarily complete. 
In this case, 
we replace 
$F$
with 
its 
completion
as a valued field. 
Since 
$(X, d)$ 
is complete, 
$\yomainmap(X)$
 is complete and hence closed in the completion of 
 $F$. 
 Algebraic independence is unchanged because any polynomial relation over 
 $K$ 
 among points of 
 $\yomainmap(X)$
  already lies in 
 $F$.
Then 
\ref{item:f6} is satisfied. 
This finishes the proof of 
Theorem 
\ref{thm:19main2}. 
\end{proof}

Letting 
$K=\yowitt{\yosmf{k}}$, 
and using 
Proposition 
\ref{prop:embedcoef}, 
we 
obtain the next corollary. 
\begin{cor}\label{cor:19main11}
Let 
$\yoratio\in (1, \infty)$, 
$(q, p)\in \yoqpset$
with 
$q\neq p$, 
and let
 $\yosmf{k}$ 
 be a perfect field of characteristic
  $p$,
$\yorepcard$ 
be a cardinal, 
and 
$G\in \yogset$. 
Put 
$R=\{\, \yoratio^{-g}\mid g\in G\sqcup\{\infty\}\,\}$. 
Then there exists a  valued field 
$(L, v)$ 
such that: 
\begin{enumerate}[label=\textup{(\arabic*)}]

\item 
the field 
$(L, v)$ 
is a
valued field   extension of 
$(\yowitt{\yosmf{k}}, \yowiv{\yosmf{k}})$;

\item 
the absolute value 
$\yonabs{v}{\yoratio}{*}$ 
is 
$R$-valued;

\item 
for each 
$R$-valued 
ultrametric 
$(X, d)$ 
with either 
$X=\emptyset$ 
or 
$\card(X\sqcup\{\ast\})\le \yorepcard$ 
for some 
$\ast\notin X$,
there exists 
an isometric
embedding 
$\yomainmap\colon (X, d)\to (L, \yonabs{v}{\yoratio}{*})$
such that 
the set 
$\yomainmap(X)$ 
is algebraically independent over 
$\yowitt{\yosmf{k}}$. 
\end{enumerate}

Moreover, 
for every 
$R$-valued 
ultrametric space
$(X, d)$, 
there exist a valued field 
$(F, u)$ 
and a map 
$\yomainmap\colon X\to F$
such that 
\begin{enumerate}[label=\textup{(\arabic*)}]

\item 
the map 
$\yomainmap\colon (X, d)\to (F, \yonabs{u}{\yoratio}{*})$ is an isometric embedding;

\item 
the field 
$(F, u)$ 
is a valued field   extension of 
$(\yowitt{\yosmf{k}}, \yowiv{\yosmf{k}})$;

\item 
$u(F)\yosub G\sqcup \{\infty\}$;

\item 
$\yomainmap(X)$ 
is closed in 
$F$;

\item 
$\yomainmap(X)$ 
is algebraically independent 
over 
$\yowitt{\yosmf{k}}$;

\item 
if 
$(X, d)$
is complete, 
then 
$F$
can be chosen to be 
complete. 

\end{enumerate}
\end{cor}
\begin{proof}
The proof is the same  as
that of  Theorem 
\ref{thm:19main2}. 
Note that, since 
$\yowitt{\yosmf{k}}$ 
is naturally a subset of 
$\youhaq{G}{\yosmf{k}}{0}{p}$
(see Proposition 
\ref{prop:embedcoef}), 
we do not need to use Proposition 
\ref{prop:pembed}. 
Thus, 
the assumption of 
the corollary 
does not 
require that 
$G$ 
be divisible. 
\end{proof}

\subsection{Broughan's conjecture}
Next we give  
an affirmative solution of 
Conjecture 
\ref{conj:bbb}. 
We begin with the definition of 
non-Archimedean Banach algebras. 
Let 
$(K, \lvert * \rvert)$ 
be 
a field equipped with a non-Archimedean 
absolute value, 
and 
$(B, \lVert *\rVert)$ 
be a 
normed linear space 
over a field 
$K$.
We say that 
$(B, \lVert *\rVert)$ 
is 
a
\emph{non-Archimedean Banach algebra over 
$K$}
(see 
\cite{MR0512894})
if the following conditions are fulfilled:
\begin{enumerate}

\item 
$B$
 is 
a 
 ring (not necessarily commutative and    unitary);

\item 
$(B, \lVert *\rVert)$ 
is a normed linear space over 
$(K, \lvert *\rvert)$; 
\item 
for every pair 
$x, y\in B$, 
we have 
$\lVert x+y \rVert \le \lVert x\rVert \lor \lVert y\rVert$
and 
$\lVert xy \rVert\le \lVert x\rVert \cdot \lVert y\rVert$; 
\item 
$B$ 
is complete with respect to 
$\lVert *\rVert$;

\item 
if 
$B$ 
has a unit 
$e$, 
then 
$\lVert e\rVert=1$.

\end{enumerate}
Some authors 
assume commutativity 
and the existence of a unit in 
the 
definition of  a Banach algebra 
(see, for example, 
\cite{MR0261361}
and 
\cite{MR0361697}). 
For instance, 
a valued field extension 
$(L, \lvert *\rvert_{L})$ 
of 
$(K, \lvert * \rvert_{K})$ 
is 
a (unitary) Banach algebra over 
$(K, \lvert * \rvert_{K})$.

As an application of Corollary 
\ref{cor:19main11}, 
we obtain the next theorem.

\begin{thm}\label{thm:frw}
Let 
$p$ 
be a prime, 
$\yorepcard$
 be a cardinal, 
and 
$\yosmf{k}$ 
be a perfect field 
of  
characteristic 
$p$
such that
the  transcendental
 degree of 
$\yosmf{k}$
over 
$\yogf{p}$ 
is equal to 
$\yorepcard$. 
Then every 
$H_{p}$-valued 
ultrametric 
 $(X, d)$ 
 with either 
$X=\emptyset$ 
or 
$\card(X\sqcup\{\ast\})\le \yorepcard$ 
for some 
$\ast\notin X$,
 can be isometrically embedded 
 into  
 $(\yowitt{\yosmf{k}}, \yowiv{\yosmf{k}})$. 
 In particular, 
 Conjecture \ref{conj:bbb}
  is true. 
\end{thm}
\begin{proof}
The theorem 
follows from 
Theorem 
\ref{thm:main0}. 
Theorem 
\ref{thm:frw} 
is actually an affirmative solution of 
Conjecture 
\ref{conj:bbb}
since 
for every cardinal 
$\yorepcard$, 
there exists a perfect field 
$\yosmf{k}$ 
of 
 characteristic 
$p$ 
such that 
the transcendental degree
of
 $\yosmf{k}$ 
over 
$\yogf{p}$ 
is equal to 
$\yorepcard$, 
and  since 
$\yowitt{\yosmf{k}}$ 
is a non-Archimedean  (commutative and unitary) Banach algebra 
over 
$\qq_{p}$
(see 
Proposition 
\ref{prop:embedcoef}). 
\end{proof}

\begin{rmk}
A solution of  Conjecture 
\ref{conj:bbb}
is given by 
fraction field 
$\yowitt{\yosmf{k}}$ of the ring of Witt vectors. 
\end{rmk}


\section{Isometric embeddings of generalized  ultrametric spaces}\label{sec:gen}

 In
 this section, 
 we
 briefly  
 explain 
an  analogue of
 Theorem 
  \ref{thm:19main2}
  for 
 ultrametric spaces 
 whose distances 
 take values in 
 a
general 
linearly ordered set. 
We 
 follow
 the 
 notations of 
  \cite{Ishiki2022highpower}.
 We refer the readers 
 to 
 \cite{MR1703500}, 
 \cite{MR3946544}, 
 and 
 \cite{Ishiki2022highpower}
 for 
the details of ultrametric spaces whose 
 distances take values in 
 linearly ordered sets.

 A linearly ordered set 
 $A$ 
 is said to 
 be 
 \emph{bottomed} 
 if 
 $A$ 
 has a minimum element. 
 In this case, we
 denote by 
 $\yobtm_{A}$
 the minimum element of 
 $A$.

Let us consider how 
we 
generalized 
our aforementioned arguments
to  ultrametrics taking values in 
linearly ordered sets. 
Let 
$G$ 
be a linearly ordered 
Abelian group
(not necessarily a subset of  
$\rr$). 
Wherever we use 
$-\log_{\yoratio}(y)$ 
and 
$\yoratio^{-x}$, 
it is enough to 
replace them 
with 
the 
antitone map 
$\yolga\colon (G, \le)\to (G, \le)$
defined by 
$x\mapsto -x$.
The 
greatest difficulty arises 
in 
arguments
of  Subsection 
\ref{conj:bbb},
where 
we
consider 
 the value 
 $d(\yoesp_{\lambda}, \yova_{\lambda})=
\inf_{x\in \yoesp_{\lambda}}d(x, \yova_{\lambda})$
since we 
use the Dedekind completeness
of
 $\rr$ 
 to define it. 
Let us see how we modify the arguments
using 
$d(\yoesp_{\lambda}, \yova_{\lambda})$. 
\begin{enumerate}

\item 
We can use the convex set 
$\bigcup_{a\in \yoesp_{\lambda}}(-\infty, \yolga(d(a, \yova_{\lambda})))$
with 
$m$
 in our arguments. 
We can replace 
$(-\infty, m)$
 with 
the subset 
$\bigcup_{a\in \yoesp_{\lambda}}(-\infty, \yolga(d(a, \yova_{\lambda})))$
of 
$G$.

\item 
In Case 1 in  Lemma 
\ref{lem:0dis}, 
the  definition of 
$\yosubmap_{\lambda+1}(\yova_{\lambda})$
only requires 
 the spherical completeness and 
$\{r_{i}\}_{i\in \zz_{\ge 0}}$.

\item 
In Case 2 of  Lemma 
\ref{lem:0dis}, 
essentially, we do not use 
$d(\yoesp_{\lambda}, \yova_{\lambda})$
since
 $d(\yoesp_{\lambda}, \yova_{\lambda})=d(a, \yova_{\lambda})$
for some 
$a\in \yoesp_{\lambda}$. 
\end{enumerate}
Thus
we can avoid 
the use of  
 the  Dedekind completeness
of 
$\rr$
in our theory. 
Therefore, we can generalize 
our aforementioned arguments 
for ultrametrics taking values in 
linearly ordered sets. 
In the notation of the following theorem, 
in the mixed characteristic case, 
that is,
 when 
$p>0$ 
and 
$K$ 
has characteristic 
$0$, 
we use the 
$p$-adic 
Mal'cev--Neumann field
whose distinguished copy of 
$\zz$ 
in 
the multiplicative value group is generated by 
$\lvert p\rvert$. 
Namely,  we obtain 
the following  analogue of 
Theorem 
\ref{thm:19main2}.
 \begin{thm}\label{thm:genmain}

Let 
$p$ 
be 
$0$
 or  a prime, 
$\yorepcard$ 
be a cardinal, 
and
$(K, \lvert *\rvert)$
 be a valued field
 whose residue class field
has characteristic 
$p$. 
Take 
 a bottomed  linearly ordered  set
 $R$  
 such that 
$R\setminus\{\yobtm_{R}\}$
 is 
a divisible linearly ordered Abelian 
(multiplicative) group
and 
$\{\, \lvert x\rvert \mid x\in K\,\}\subset R$. 
Then there exists a  valued field 
$(L, \lVert *\rVert)$ 
such that: 
\begin{enumerate}[label=\textup{(\arabic*)}]

\item 
the field 
$L$ 
is a valued field   extension of 
$K$;

\item 
$\{\, \lVert x\rVert\mid x\in L\, \} = R$;

\item 
for each 
$R$-valued ultrametric 
$(X, d)$ 
with either 
$X=\emptyset$ 
or 
$\card(X\sqcup\{\ast\})\le \yorepcard$ 
for some 
$\ast\notin X$,
there exists 
an isometric 
embedding 
$I\colon (X, d)\to (L, \lVert *\rVert)$
such that 
the set 
$I(X)$ 
is algebraically independent over 
$K$. 
\end{enumerate}
Moreover, 
for every 
$R$-valued
 ultrametric space
$(X, d)$, 
there exist a valued field 
$(F, \lVert*\rVert)$ 
and a map 
$\yomainmap\colon X\to F$
such that 
\begin{enumerate}[label=\textup{(\arabic*)}]
\item 
the map 
$\yomainmap\colon (X, d)\to (F, \lVert*\rVert)$ 
is an isometric embedding;

\item 
the field 
$(F, \lVert*\rVert)$ 
is a valued field   extension of 
$(K, \lvert*\rvert)$;

\item 
for
every
 $x\in F$, 
 we have
$\lVert x\rVert \in  R$;

\item 
$\yomainmap(X)$ 
is closed in 
$F$;

\item 
$\yomainmap(X)$ 
is algebraically independent 
over 
$K$;

\item 
if 
$(X, d)$
is complete
as
a uniform space, 
then 
$(F, \lVert*\rVert)$
can be chosen
to be complete. 

\end{enumerate}

 \end{thm}
 
   A \emph{generalized metric space} is 
  a set equipped with 
  a metric taking values in 
  a linearly ordered 
  Abelian group 
  $G$
  instead of 
  $\rr$. 

Arens and Eells'
original theorem 
\cite[Theorem]{MR81458}
contains the statement 
that 
every separable uniform space 
can be uniformly continuously  embedded into 
a locally convex linear space as
a closed subset. 
The following corollary 
is 
a generalization of it.

 \begin{cor}\label{cor:5252}
 Every generalized 
 metric space 
 that is not metrizable in the ordinary  sense
  can be 
uniformly continuously 
embedded into 
a valued field. 
 \end{cor}
 \begin{proof}
 In  \cite[Lemma 2.21]{Ishiki2022highpower}, 
 it is stated that 
 a generalized metric space
  that is not metrizable in the ordinary  sense
 is uniformly equivalent 
 to an ultrametric space 
 whose distances take 
 values in a linearly ordered set. 
 Combining this fact and 
 Theorem
  \ref{thm:genmain}, 
we observe that 
Corollary 
 \ref{cor:5252}
is 
true. 
 \end{proof}
 

\section{Questions}\label{sec:ques} 


As a  sophisticated version of 
a non-Archimedean 
Arens--Eells 
theorem 
for linear spaces, 
we ask the next question concerning 
ranges of ultrametrics. 
\begin{ques}
Let 
$R$, 
$T$ 
be  range sets, 
and 
$S$
 be a range set such that 
$S\setminus \{0\}$ 
is  a multiplicative subgroup of 
$\rr_{> 0}$. 
Take a non-Archimedean valued field  
$(K, \lvert*\rvert_{K})$ 
and 
assume that 
\begin{enumerate}

\item 
$R\cdot S=\{\, r\cdot s\mid r\in R, s\in S\, \}\yosub T$; 
\item 
$\lvert*\rvert_{K}$ 
is
 $S$-valued. 
\end{enumerate}
For every 
$R$-valued 
ultrametric space 
$(X, d)$, 
do there exist
a 
$T$-valued 
ultra-normed  linear space  
$(V, \lVert *\rVert_{L})$ 
over 
$(K, \lvert*\rvert_{K})$, 
and 
an isometric map 
$I\colon X\to V$?
\end{ques}


As a counterpart of 
Theorem 
\ref{thm:19main2}, 
we ask
whether 
any valued field with a
sufficiently large 
quotient residue field
admits 
isometric embeddings
from 
ultrametrics. 
\begin{ques}
Let 
$p$ 
be 
$0$ 
or  a prime, 
  $\yosmf{k}$
   be an infinite  field of characteristic 
   $p$, 
  and 
  $R$ be a range set such that 
  $R\setminus\{0\}$ 
  is a multiplicative 
  subgroup of 
  $\rr_{>0}$. 
Take  an 
arbitrary
  complete  valued field 
$(K, \lvert*\rvert_{K})$
with the residue class field 
$\yosmf{k}$
such that 
$\{\, \lvert x\rvert_{K}\, \mid x\in K\,  \}=R$ 
and assume that 
for some 
$\yoratio \in (1, \infty)$, 
$(K, \lvert*\rvert_{K})$
is a valued field extension of 
$(\qq, \yonabs{v_{\mathrm{triv}}}{\yoratio}{*})$ 
if $p=0$, where 
$v_{\mathrm{triv}}$ 
denotes the trivial valuation on 
$\qq$, 
and of 
$(\qq_{p}, \yonabs{v_{p}}{\yoratio}{*})$
if 
$p>0$.
For every 
$R$-valued  ultrametric  
space 
$(X, d)$ 
with 
$\card(X\sqcup\{\ast\})\le\card(\yosmf{k})$
for some 
$\ast\notin X$, 
does there exist an isometric embedding 
$I\colon X\to K$?
\end{ques}


\bibliographystyle{myplaindoidoi}
\bibliography{../../../bibtex/UU.bib}


\end{document}